\documentclass[12pt]{amsart}
\usepackage[T1]{fontenc}
\usepackage{mathptmx}

\usepackage{comment}
\usepackage[colorlinks=true,linkcolor=magenta,citecolor=blue]{hyperref}     
\numberwithin{equation}{section}
\usepackage{mathtools}
\usepackage{amsmath, amssymb,amsthm}
\usepackage[capitalise]{cleveref}   
\usepackage{array}
\usepackage{graphicx}
\usepackage{float}
\usepackage{caption,subcaption}
\usepackage{tikz}
\usepackage{tikz-cd}
\usetikzlibrary{matrix,arrows,arrows.meta}
\usepackage[mathscr]{euscript}
\usepackage{xcolor}

\def\Cbb{\mathbb{C}}

 \def\Pbb{\mathbb{P}}
\def\Qbb{\mathbb{Q}}

 \def\Zbb{\mathbb{Z}}

 \def\Bcal{\mathcal{B}}
  
\def\Ecal{\mathcal{E}} \def\Fcal{\mathcal{F}}
 
\def\Ical{\mathcal{I}}

\def\Ocal{\mathcal{O}} 
 \def\Rcal{\mathcal{R}}
 
\def\Ucal{\mathcal{U}}

\def\Cfrak{\mathfrak{C}}

 \def\Tfrak{\mathfrak{T}}

\newtheorem{thm}{Theorem}[section]
\newtheorem{mainthm}{Theorem}

\newtheorem{lem}[thm]{Lemma}

\theoremstyle{definition}
\newtheorem{rem}[thm]{Remark}
\newtheorem{defi}[thm]{Definition}
\theoremstyle{plain}
\newtheorem{cor}[thm]{Corollary}
\newtheorem{prop}[thm]{Proposition}

\crefname{thm}{Theorem}{Theorems}
\Crefname{thm}{Theorem}{Theorems}
\crefname{lem}{Lemma}{Lemmas}
\Crefname{lem}{Lemma}{Lemmas}
\crefname{prop}{Proposition}{Propositions}
\Crefname{prop}{Proposition}{Propositions}
\crefname{cor}{Corollary}{Corollaries}
\Crefname{cor}{Corollary}{Corollaries}
\crefname{defi}{Definition}{Definitions}
\Crefname{defi}{Definition}{Definitions}
\crefname{rem}{Remark}{Remarks}
\Crefname{rem}{Remark}{Remarks}

\newcommand\supp{\mathrm{supp}}

\DeclareMathOperator{\Poly}{Poly}
\DeclareMathOperator{\End}{End}
\DeclareMathOperator{\id}{id}

\DeclareMathOperator{\Supp}{Supp}
\DeclareMathOperator{\Sing}{Sing}

\DeclareMathOperator{\Aut}{Aut}

\DeclareMathOperator{\PGL}{PGL}
\DeclareMathOperator{\GL}{GL}
\DeclareMathOperator{\Sym}{Sym}
\DeclareMathOperator{\im}{im}

\newcommand{\Desc}{\operatorname{Desc}}

\title{A general endomorphism of $\mathbb{P}^k$ has trivial iterated centralizer  }

\author{Yugang Zhang}
\address{Yugang Zhang, Universit\'e Bourgogne Europe, CNRS, IMB UMR 5584,
21000 Dijon, France}
\email{yugang.zhang@ube.fr}
\urladdr{https://sites.google.com/view/yugangzhang}
\thanks{The author acknowledges support from the ATRACT programme of the
R\'egion Bourgogne--Franche--Comt\'e through the project ADYAUS
(RECH-ATRAC-000012), and from the French National Research Agency through the
project DynAtrois (ANR-24-CE40-1163).}

\begin{document}

\begin{abstract}
Fix integers $k,d\geq2$. Let $\End_d^k$ denote the parameter space
of endomorphisms of $\Pbb^k$ of algebraic degree $d$.
We prove that there exists a dense Zariski open subset
$U_{d,k}\subset\End_d^k$, defined over $\Qbb$, such that, for every
$f\in U_{d,k}(\Cbb)$, every dominant rational self-map of $\Pbb^k$
commuting with an iterate of $f$ is itself an iterate of $f$.
The key intermediate result classifies dominant rational
semiconjugacies to iterates of $f$ from normal projective varieties
of dimension $k$. The proof uses finite-level monodromy and its action
on the rooted preimage tree of a point outside the branch locus of
an iterate of $f$.
An analogous centralizer theorem holds generically for
regular polynomial endomorphisms of $\Cbb^k$.
These results extend several theorems of Pakovich to higher dimensions.

We give two applications. We first classify periodic dominant
correspondences between a general endomorphism and an arbitrary
endomorphism of degree at least two. We then show that, for a general
$f\in\End_d^k$, an irreducible hypersurface of $\Pbb^k$ is $f$-special
in the sense of Ghioca--Tucker and DeMarco--Mavraki if and only if
it is $f$-preperiodic.
\end{abstract}
\maketitle

\setcounter{tocdepth}{1}
\tableofcontents

\section{Introduction}

\subsection{Main results and applications}

\subsubsection{Main results}
Throughout the paper, we fix integers $k,d\geq2$, where $k$ is the
dimension and $d$ is the algebraic degree. We work over $\Cbb$.
For $e\geq1$, let $\operatorname{Rat}_e^k$ denote the parameter space of
dominant rational self-maps of $\Pbb^k$ of algebraic degree $e$, and let
$\End_e^k:=\End_e(\Pbb^k)\subset\operatorname{Rat}_e^k$
denote the locus of endomorphisms.

Set $N_{e,k}:=(k+1)\binom{k+e}{e}-1$.
After choosing homogeneous coordinates, $\operatorname{Rat}_e^k$ is identified
with the Zariski open subset of $\Pbb^{N_{e,k}}$ consisting of
$(k+1)$-tuples of homogeneous forms of degree $e$, without a common
factor and defining a dominant map, up to a common scalar.
Such a tuple defines an endomorphism precisely when its Macaulay
resultant does not vanish. Thus $\End_e^k$ is a dense affine Zariski
open subset of $\operatorname{Rat}_e^k$, of dimension $N_{e,k}$ and
defined over $\Qbb$; see \cite{DSbook,LevyMorphisms}.
For $k=1$, see also \cite{SilvermanRatMaps}.
Set
\[
 \operatorname{Rat}(\Pbb^k):=\bigsqcup_{e\geq1}\operatorname{Rat}_e^k,
 \qquad
 \End^k:=\bigsqcup_{e\geq1}\End_e^k.
\]

For $f\in\End^k$, define its \emph{(rational) centralizer} and its
\emph{(rational) iterated centralizer} by
\[
 C_{\mathrm{rat}}(f)
 :=\{h\in\operatorname{Rat}(\Pbb^k)\mid h\circ f=f\circ h\},
 \qquad
 C_{\mathrm{rat}}(f^\infty)
 :=\bigcup_{m\geq1}C_{\mathrm{rat}}(f^m).
\]
We write $\langle f\rangle:=\{f^s\mid s\geq0\}$, where
$f^s$ denotes the $s$-th iterate and $f^0=\id$.

Our main result states that the iterated centralizer of a
general endomorphism of degree at least two is trivial in the
following precise sense.

\begin{mainthm}\label{thm:main}
There exists a dense Zariski open subset
$U_{d,k}\subset\End_d^k$, defined over $\Qbb$, such that, for every
$f\in U_{d,k}(\Cbb)$,
\[
 C_{\mathrm{rat}}(f^\infty)=\langle f\rangle.
\]
\end{mainthm}

\begin{rem}
Since $U_{d,k}$ is defined over $\Qbb$, the same conclusion holds
after a base change to any algebraically closed field of
characteristic zero by the Lefschetz principle.
\end{rem}

Some genericity is necessary. For example, the power map
$p_d:[X_0:\cdots:X_k]\mapsto[X_0^d:\cdots:X_k^d]$ commutes with
every other power map and therefore its iterated centralizer is strictly larger
than $\langle p_d\rangle$. Moreover, for every $f$ and every
$N\geq2$, one has $f\in C_{\mathrm{rat}}(f^N)$ but
$f\notin\langle f^N\rangle$.

We work with a class of endomorphisms, called \emph{simple}, defined
by explicit conditions on their ramification and postcritical behavior
(Definition~\ref{def:simple}). A simple endomorphism satisfying an
additional algebraic symmetry condition is called \emph{strongly simple}
(Definition~\ref{defi:strongly-simple}). The strongly simple locus contains
a dense Zariski open subset defined over $\Qbb$, and the conclusion of
Theorem~\ref{thm:main} holds for every strongly simple endomorphism.

The one-dimensional analogue of Theorem~\ref{thm:main} for a general
rational map follows from Pakovich's
work~\cite{pakovich-deg4,pakovich2026periodiccurvesgeneralendomorphisms}.
Inspired by his work, our key intermediate result is the following
semiconjugacy rigidity theorem, extending
\cite[Theorem~1.4]{pakovich-deg4} and
\cite[Theorem~5.8]{pakovich2026periodiccurvesgeneralendomorphisms}
to higher dimensions.

\begin{mainthm}\label{thm:main-semiconjugacy}
Let $f\in\End_d^k$ be simple, and let $m\geq1$.
Let $Y$ be a normal integral projective variety of dimension $k$, and let
\[
 h:Y\dashrightarrow\Pbb^k,\qquad u:Y\dashrightarrow Y
\]
be dominant rational maps satisfying
\[
 f^m\circ h=h\circ u.
\]
Then there are a unique integer $s\geq0$ and a unique birational map
$\mu:Y\dashrightarrow\Pbb^k$ such that
\[
 h=f^s\circ\mu,\qquad u=\mu^{-1}\circ f^m\circ\mu.
\]
Moreover, if $h$ is a finite morphism, then $\mu$ is an isomorphism.
If $Y=\Pbb^k$ and $u$ is a morphism, then
$\mu\in\PGL_{k+1}(\Cbb)$; in particular, $h$ is a morphism.
\end{mainthm}

In particular, if $f$ is simple, $h\in\operatorname{Rat}(\Pbb^k)$ is birational
and $h\circ f=g\circ h$ for some $g\in\End^k$, then
$h\in\PGL_{k+1}(\Cbb)$. This gives birational conjugacy rigidity for
simple endomorphisms in arbitrary dimension, and the conjugacy is
necessarily linear; compare Cantat--Xie
\cite[Theorem~A]{CantatXieBirationalConjugacies} in dimension two.

Applied with $Y=\Pbb^k$ and $u=f^m$, Theorem~\ref{thm:main-semiconjugacy} gives
\[
 h\circ f^m=f^m\circ h
 \quad\Longrightarrow\quad
 h=f^s\circ\alpha,\qquad \alpha\circ f^m=f^m\circ\alpha,
 \qquad \alpha\in\PGL_{k+1}(\Cbb).
\]
For strongly simple $f$, the additional symmetry condition forces
$\alpha=\id$ for every $m\geq1$, giving the conclusion of
Theorem~\ref{thm:main}.

As in Pakovich's work~\cite{pakovich-deg4,pakovich2026periodiccurvesgeneralendomorphisms},
decomposition rigidity is a key ingredient in our proof of semiconjugacy
rigidity. Among other ingredients, his proofs draw on genus estimates based on the
Riemann--Hurwitz formula and orbifold methods on Riemann surfaces,
with arguments tailored to dimension one.
Our proof studies the monodromy action on the rooted preimage trre
of a point outside the branch locus of an iterate of $f$.
The key step is to show that the monodromy orbit of a pair of
distinct leaves is determined precisely by the first iterate under
which the two leaves have the same image. We prove this by combining
local monodromy around the postcritical hypersurfaces with the
independent action of suitable subgroups on distinct subtrees.
Consequently, every monodromy-invariant partition of the leaves at
depth $m$ is the partition into fibers of $f^j$ for some $0\leq j\leq m$.
This gives the higher-dimensional decomposition rigidity theorem
(Theorem~\ref{thm:sc-normal-factors}).
To show that the simplicity conditions underlying our
argument hold on a dense Zariski open subset of $\End_d^k$, we use,
in particular, the theorem of
Gauthier--Taflin--Vigny~\cite{GauthierTaflinVignyPCF} that
postcritically finite (PCF) endomorphisms are not Zariski dense,
in sharp contrast with the situation in dimension
one~\cite[Theorem~A]{DeMarcoPCFDensity}.
See Section~\ref{subsec:strategy} for an outline of the proof of
Theorem~\ref{thm:main-semiconjugacy}.

\medskip

Following Sibony~\cite{sibonybook}, a polynomial endomorphism
$f:\Cbb^k\to\Cbb^k$ of degree $d$ is \emph{regular} if it extends
holomorphically to an endomorphism of $\Pbb^k$.
Equivalently, the resultant of the leading homogeneous parts of the
components of $f$ is nonzero. By abuse of notation, we
continue to denote the projective extension by $f$.

No regular polynomial endomorphism is simple, since the hyperplane
at infinity is a totally invariant component of its ramification
divisor. Nevertheless, the same centralizer conclusion holds
generically within this family.

Let $\Poly_d^k$ denote the parameter space of such maps.
Thus $\Poly_d^k$ is a dense affine
Zariski open subset, defined over $\Qbb$, of an affine coefficient
space of dimension $k\binom{k+d}{d}$.

\begin{mainthm}\label{thm:regular-polynomial-centralizer}
There exists a dense Zariski open subset
$U^{\mathrm{poly}}_{d,k}\subset\Poly_d^k$, defined over $\Qbb$,
such that, for every $f\in U^{\mathrm{poly}}_{d,k}(\Cbb)$,
\[
 C_{\mathrm{rat}}(f^\infty)=\langle f\rangle.
\]
\end{mainthm}

Its proof combines Theorem~\ref{thm:main} with the image--fiber
principle and measure rigidity from the author's earlier
work~\cite{ZhangMeasureRigidity}. 

\subsubsection{Applications}
As in~\cite{pakovich-deg4,pakovich2026periodiccurvesgeneralendomorphisms},
we classify $k$-dimensional periodic algebraic correspondences in
$\Pbb^k\times\Pbb^k$.
Their classification has played an important role in arithmetic
and complex dynamics. Medvedev--Scanlon~\cite{MedvedevScanlon2014}
classified invariant subvarieties under products of polynomial maps
in one variable. Pakovich extended the classification of invariant curves
from polynomial maps to nonspecial rational
maps~\cite{PakovichInvariantCurves}. He also proved that, for
general rational maps of the same degree, the existence of a
periodic curve dominating both factors forces the maps to be
conjugate~\cite{pakovich-deg4,pakovich2026periodiccurvesgeneralendomorphisms}.
This implication is an ingredient in Ji--Xie's proof that the
multiplier spectrum morphism is generically injective
\cite{JiXieMultiplierSpectrum}.

A \emph{dominant correspondence} is an integral subvariety
$Z\subset\Pbb^k\times\Pbb^k$ of dimension $k$ whose two projections
are dominant. For $\Phi=f\times g$, it is \emph{periodic} if
$\Phi^m(Z)=Z$ for some $m\geq1$.

\begin{mainthm}
\label{thm:periodic-correspondences}
Let $f\in U_{d,k}(\Cbb)$ be strongly simple, and let
$g:\Pbb^k\to\Pbb^k$ be any endomorphism of degree at least two.
Let $Z\subset\Pbb^k\times\Pbb^k$ be a dominant correspondence.
The correspondence $Z$ is periodic under $f\times g$ if and only if
it has one of the forms
\begin{equation}\label{eq:periodic-correspondence-graphs}
 Z=\bigl\{(x,y)\mid y=(\alpha\circ f^r)(x)\bigr\}
 \quad\text{or}\quad
 Z=\bigl\{(x,y)\mid x=(f^r\circ\alpha^{-1})(y)\bigr\},
\end{equation}
where $r\geq0$ and $\alpha\in\PGL_{k+1}(\Cbb)$ satisfies
$g=\alpha\circ f\circ\alpha^{-1}$.
In particular, every periodic dominant correspondence is invariant under $f\times g$.
\end{mainthm}

In particular,for the diagonal $\Delta:=\{(x,x)\mid x\in\Pbb^k\}$, the
classification gives
\[
 \Delta\text{ is preperiodic under }f\times g
 \quad\Longleftrightarrow\quad f=g;
\]
see Corollary~\ref{cor:preperiodic-diagonal}.
This equivalence fails in dimension one. Indeed, every quadratic $f:\Pbb^1\to\Pbb^1$ admits
a nontrivial involution $\sigma\in\PGL_2(\Cbb)$ with
$\sigma^2=\id$ and $f\circ\sigma=f$.
Set $g:=\sigma\circ f\ne f$ and
$\Gamma_\sigma:=\{(x,\sigma(x))\mid x\in\Pbb^1\}$.
Then
\[
 (f\times g)(\Delta)=\Gamma_\sigma,
 \qquad (f\times g)(\Gamma_\sigma)=\Gamma_\sigma,
\]
so $\Delta$ is preperiodic under $f\times g$.

Pakovich's one-dimensional classification
is stated for two general rational maps $f$ and $g$.
To treat arbitrary $g$, we classify dominant rational dynamical
quotients of $f$ (Theorem~\ref{thm:rational-dynamical-quotients}).

\bigskip
We next give an application of Theorem~\ref{thm:main} that was
pointed out to the author by Johan Taflin.

Let $f\in\End_d^k$, and let $X\subset\Pbb^k$ be an irreducible
subvariety. We say that $X$ is \emph{$f$-preperiodic} if
$f^a(X)=f^b(X)$ for some integers $a>b\geq0$.
Following the formulation of DeMarco--Mavraki
\cite{DeMarcoMavrakiPreperiodic} of the notion introduced by
Ghioca--Tucker~\cite{GhiocaTucker2021}, we say that $X$ is
\emph{$f$-special} if there exist an irreducible subvariety
$Z\subset\Pbb^k$ containing $X$, an integer $n\geq1$, and a
polarizable endomorphism $\psi:Z\to Z$ such that
\[
 f^n(Z)=Z,\qquad
 \left.f^n\right|_Z\circ\psi=\psi\circ\left.f^n\right|_Z,
 \qquad X\text{ is }\psi\text{-preperiodic}.
\]
Here polarizability means that there exist an ample line bundle $L$
on $Z$ and an integer $q>1$ such that
$\psi^*L\simeq L^{\otimes q}$.

\begin{cor}\label{cor:special-preperiodic}
Let $f\in U_{d,k}(\Cbb)$, where $U_{d,k}$ is the dense Zariski open
subset supplied by Theorem~\ref{thm:main}. Then an irreducible
hypersurface $X\subset\Pbb^k$ is $f$-special if and only if it is
$f$-preperiodic.
\end{cor}

\begin{proof}
Every $f$-preperiodic subvariety is $f$-special: one may take
$Z=\Pbb^k$, $\psi=f$, and $n=1$.
Conversely, suppose that an irreducible hypersurface $X\subset\Pbb^k$
is $f$-special, and choose $Z,\psi,n$ as in the definition.
Since $X\subset Z\subset\Pbb^k$ and $\dim X=k-1$, either $Z=X$
or $Z=\Pbb^k$. In the first case, $f^n(X)=X$, so $X$ is
$f$-periodic. In the second case, $\psi$ is an endomorphism of
$\Pbb^k$ commuting with $f^n$. By Theorem~\ref{thm:main},
$\psi=f^r$ for some $r\geq1$. Thus preperiodicity under $\psi$
implies preperiodicity under $f$.
\end{proof}

For such $f$, this shows that the conjecture of DeMarco--Mavraki
\cite[Conjecture~4.1]{DeMarcoMavrakiPreperiodic} and that of Cantat
and DeMarco--Mavraki \cite[Conjecture~2.7]{DeMarcoCIRMLectures}
are equivalent for hypersurfaces. For further details, we refer the
reader to the recent work of Gauthier--Taflin~\cite{GauthierTaflin2026},
where the latter conjecture is proved for endomorphisms of $\Pbb^2$.

\subsection{Further context and related work}
A classical rigidity problem in dynamics asks whether a typical
system has the smallest possible centralizer. Smale explicitly
formulated this question for $C^r$-diffeomorphisms
\cite{Smale1998}. For a closed smooth manifold $M$, he asked
whether, for every $r\geq1$, the set of
$T\in\operatorname{Diff}^r(M)$ whose centralizer is
$\langle T\rangle=\{T^n\mid n\in\Zbb\}$ is dense in the $C^r$
topology. Bonatti, Crovisier, and Wilkinson proved the
$C^1$-generic (in the Baire sense) form of this statement: if $M$ is
connected, such diffeomorphisms form a residual subset of
$\operatorname{Diff}^1(M)$
\cite{BonattiCrovisierWilkinson2009}.

For rational maps of the projective line, the study of commuting maps
goes back to Julia~\cite{Julia1922}, Fatou~\cite{Fatou1923}, and
Ritt~\cite{Ritt1923}. Ritt proved that two commuting rational maps of
degree at least two either have a common iterate or are both
\emph{special}~\cite{Ritt1923}. Here a rational map $F$ of degree
$d\geq2$ is called special if it is conjugate to $z\mapsto z^{\pm d}$
or to a Chebyshev map $\pm T_d$, or if it is a Latt\`es map
\cite{MilnorLattes}; otherwise it is called \emph{nonspecial}.
Eremenko later gave a new proof of this theorem using modern methods
of complex dynamics~\cite{Eremenko1990}. For a nonspecial rational
map $F$, Pakovich proved that the centralizer $C_{\mathrm{rat}}(F)$ is virtually
cyclic: there exist finitely many $X_1,\ldots,X_s\in C_{\mathrm{rat}}(F)$ such that
$C_{\mathrm{rat}}(F)=\{X_i\circ F^n\mid 1\leq i\leq s,\ n\geq0\}$
\cite[Theorem~1.2]{Pak20}; see \cite{Pak21} for a finer structural
description. Beaumont recently extended virtual cyclicity to the
iterated centralizer: $C_{\mathrm{rat}}(F^\infty)$ is virtually cyclic for every
nonlinear nonspecial rational map, and
$C_{\mathrm{rat}}(F^\infty)=C_{\mathrm{rat}}(F^N)$ for some $N\geq1$ whenever $F$ is not conjugate
to a power map
\cite[Theorem~1.3 and Corollary~1.4]{beaumont2025centralizersendomorphismsprojectiveline}.

In higher dimensions, Dinh--Sibony proved that commuting endomorphisms
$f,g$ of $\Pbb^k$ with degrees at least two and
$(\deg f)^a\neq(\deg g)^b$ for all $a,b\geq1$ are induced by affine
maps on a common orbifold; in particular, both maps are PCF
\cite[Theorem~1.1, Corollaire~1.3]{DinhSibony2002}.
In dimension two, Kaufmann classified the complementary case in which
the degrees are multiplicatively dependent but the two maps have no
common iterate: the pair consists either of Latt\`es endomorphisms of
$\Pbb^2$ or of homogeneous polynomial endomorphisms of $\Cbb^2$
inducing Latt\`es maps on the line at infinity
\cite[Theorem~1.1]{Kaufmann2018}.
See also Dinh's work~\cite{DinhCommuting}.
Together, these results imply that commuting endomorphisms of
$\Pbb^2$ of degrees at least two without a common iterate are PCF.
Silverman studied the affine automorphism groups of the
two-dimensional folding maps; the $A_2$- and $B_2$-families yield
explicit regular polynomial endomorphisms $F_n$ satisfying
$C_{\mathrm{rat}}(F_n)\supsetneq\langle F_n\rangle$
\cite{SilvermanCommutingAffinePlane}.

More recently, the author used measure rigidity and the image--fiber
principle to study iterated centralizers of regular polynomial
endomorphisms of $\Cbb^2$. In that result, centralizers are taken
among regular polynomial endomorphisms. If $f$ is not affinely
conjugate to a homogeneous map and the endomorphism $f_\Pi$ induced
on the line at infinity is not conjugate to a power map, then its
iterated centralizer stabilizes at the centralizer of $f^N$ for some
$N\geq1$; if, in addition, $f_\Pi$ is nonspecial, then the iterated
centralizer is of the form $\mathcal F\circ\langle f\rangle$ for a finite
set $\mathcal F$ \cite[Theorem~1.8]{ZhangMeasureRigidity}.
We also use the generic measure rigidity theorem of that work
\cite[Theorem~1.4]{ZhangMeasureRigidity} to prove
Theorem~\ref{thm:regular-polynomial-centralizer}.

\subsection{Outline of the proof of Theorem~\ref{thm:main-semiconjugacy}}
\label{subsec:strategy}
Let $\Rcal_f$ be the ramification divisor of $f$ and $R_f$ its support.
We regard $\Rcal_f$ as a closed subscheme.
Put $B_f:=f(R_f)$. We say that $f$ is simple if it satisfies the following four
conditions (Definition~\ref{def:simple}):
\begin{enumerate}
\item[\textup{(S1)}] $\Rcal_f$ is integral;
\item[\textup{(S2)}] $f|_{R_f}:R_f\to B_f$ is birational;
\item[\textup{(S3)}]
$B_f\subsetneq f^{-1}\bigl(f(B_f)\bigr)$;
\item[\textup{(S4)}] $f$ is not postcritically finite (i.e.,
$\bigcup_{i\geq1}f^i(R_f)$ is not Zariski closed).
\end{enumerate}
Proposition~\ref{prop:simple-open} shows that these conditions hold
on a dense Zariski open subset defined over $\Qbb$.
This definition is inspired by the work of
Ingram--Ramadas--Silverman~\cite{pcfsilverman}.

For the remainder of this subsection, we assume that $f$ is simple.
Fix an integer $m\geq1$. For $i\geq0$, put $B_i:=f^i(B_f)$, and set
$D_m:=B_0\cup\cdots\cup B_{m-1}$ and
$U_m:=\Pbb^k\setminus D_m$. Fix $y\in U_m$. The vertices at depth
$s$ of the rooted preimage tree are $T_s(y):=f^{-s}(y)$; see
Definition~\ref{def:tree}. The group $G_m(y)$ denotes the monodromy
group acting on its leaves $T_m(y)$. By \textup{(S1)} and
\textup{(S2)}, Lemma~\ref{lem:generic-fold} shows that, at a general
point of $R_f$, the map $f$ is analytically a fold
$(x_1,\ldots,x_k)\mapsto(x_1,\ldots,x_{k-1},x_k^2)$; in particular,
its local monodromy is a transposition.

\begin{figure}[!htbp]
  \centering
  \resizebox{0.8\textwidth}{!}{%
  \begin{tikzpicture}[
    every node/.style={font=\small},
    blu/.style={circle,draw=blue,fill=blue!6,text=blue,
      minimum size=8mm,inner sep=0pt},
    gry/.style={circle,draw=gray,fill=gray!20,
      minimum size=8mm,inner sep=0pt},
    edge/.style={draw=black,thin},
    pair/.style={<->,>={Stealth[length=5pt]},red,thin}
  ]
    \node[gry,fill=white,draw=black] (y) at (-0.5,0) {$y$};

    \node[blu] (v)  at (-4,-2) {$v$};
    \node[gry] (a1) at (-0.5,-2) {};
    \node[gry] (a2) at ( 1.5,-2) {};
    \node[gry] (a3) at ( 3.5,-2) {};

    \node[blu] (up) at (-4,-4) {$u_+$};
    \node[gry] (b1) at (-1.5,-4) {};
    \node[gry] (b2) at ( 1.5,-4) {};
    \node[blu] (um) at ( 4,-4) {$u_-$};

    \foreach \i/\x in {1/-7,2/-5,3/-3,4/-1}
      \node[blu] (p\i) at (\x,-6.5) {$x_\i^+$};
    \foreach \i/\x in {1/1,2/3,3/5,4/7}
      \node[blu] (m\i) at (\x,-6.5) {$x_\i^-$};

    \foreach \t in {v,a1,a2,a3} \draw[edge] (y) -- (\t);
    \foreach \t in {up,b1,b2,um} \draw[edge] (v) -- (\t);
    \foreach \i in {1,2,3,4} \draw[edge] (up) -- (p\i);
    \foreach \i in {1,2,3,4} \draw[edge] (um) -- (m\i);

    \draw[pair] (up.north) to[bend left=18] (um.north);
    \foreach \i/\b in {1/22,2/32,3/42,4/52}
      \draw[pair] (p\i.south) to[bend right=\b] (m\i.south);

    \node[anchor=east] at (-8.3,0) {depth $0$};
    \node[anchor=east] at (-8.3,-2) {depth $1$};
    \node[anchor=east] at (-8.3,-4) {depth $m-j=2$};
    \node[anchor=east] at (-8.3,-6.5) {depth $m=3$};

    \node at (-0.5,-10) {$
      (x_1^+\,x_1^-)(x_2^+\,x_2^-)
      (x_3^+\,x_3^-)(x_4^+\,x_4^-)$};
  \end{tikzpicture}%
  }
  \caption{Monodromy around $B_1$ for $j=1$, $m=3$, and $k=d=2$.
  Red arrows indicate exchanged vertices. The permutation below the
  tree describes the action on the highlighted leaves; only one
  affected subtree is shown.}
  \label{fig:meridian-j1-m3}
\end{figure}
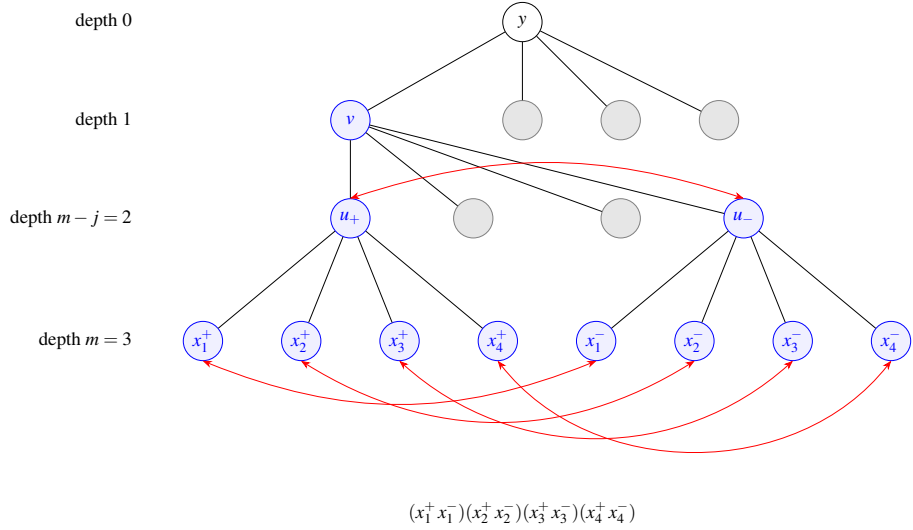

Here a \emph{meridian} around a branch component means a based loop
obtained by going from $y$ to a small transverse disk at a general
smooth point of the component, circling it once positively, and
returning to $y$; see Definition~\ref{def:meridian}.

Fix $0\leq j<m$ and put $r:=m-j-1$. Around a general point of the
branch component $B_r$, the monodromy of a meridian on $T_m(y)$ is a
product of disjoint transpositions: at each relevant ramification
point, the factorization
\[
 f^m=f^r\circ f\circ f^j
\]
is locally \'etale--fold--\'etale by \textup{(S4)}; see
Lemmas~\ref{lem:ram-level-separation} and~\ref{lem:generic-fold}.
The two leaves exchanged by each transposition first coalesce
after $j+1$ iterations:
\[
 f^{j+1}(x)=f^{j+1}(x'),\qquad f^j(x)\ne f^j(x').
\]
As $j$ varies, the monodromy permutations induced by these meridians
normally generate $G_m(y)$.
Figure~\ref{fig:meridian-j1-m3} illustrates this monodromy action
for $j=1$ and $m=3$.

The key point is to control how these permutations act
throughout the tree. At depth one, $G_1(y)=S_{d^k}$.
For $m\geq2$, we study the kernel
\[
 K_m:=\ker\bigl(G_m(y)\to G_1(y)\bigr).
\]
Condition \textup{(S3)} provides elements of $K_m$ that move leaves
below some vertices of $T_1(y)$ and fix all leaves below the others.
Using these elements and their conjugates, we prove by induction
that, for any distinct $v,w\in T_1(y)$, $K_m$ acts transitively on
pairs consisting of one leaf below $v$ and one below $w$
(Lemma~\ref{lem:two-coordinate-extraction}). As a consequence, each set
\[
 \Omega_j:=\bigl\{\{x,x'\}\subset T_m(y)\mid
 f^j(x)=f^j(x'),\quad f^{j-1}(x)\ne f^{j-1}(x')\bigr\},
 \qquad 1\leq j\leq m,
\]
is a single $G_m(y)$-orbit.

This gives the classification of invariant partitions in
Proposition~\ref{prop:sc-monodromy-partitions}: for every
$G_m(y)$-invariant equivalence relation $\sim$ on $T_m(y)$, there
is a unique integer $0\leq j\leq m$ such that
\[
 x\sim x'\quad\Longleftrightarrow\quad f^j(x)=f^j(x').
\]
To see this, suppose that $\sim$ is not the equality relation,
and choose the largest $j$ such that $x\sim x'$ for some
$\{x,x'\}\in\Omega_j$. Since $\Omega_j$ is a single
orbit, invariance implies that every pair in $\Omega_j$ is in this
relation. Transitivity then gives $x\sim x'$ whenever
$f^j(x)=f^j(x')$. Conversely, if $f^j(x)\ne f^j(x')$, then
$\{x,x'\}\in\Omega_t$ for some $t>j$, so $x\not\sim x'$ by
the maximality of $j$.

This classification
gives the higher-dimensional decomposition theorem
(Theorem~\ref{thm:sc-normal-factors}). Suppose that
\[
 f^m=P\circ Q,\qquad Q:\Pbb^k\dashrightarrow Z,
 \qquad P:Z\dashrightarrow\Pbb^k,
\]
where $Z$ is normal, integral and projective and both maps are
dominant rational. Then there are an integer $0\leq j\leq m$ and
a birational map $\alpha:Z\dashrightarrow\Pbb^k$ such that
\[
 Q=\alpha^{-1}\circ f^j,\qquad P=f^{m-j}\circ\alpha.
\]
If both factors are finite, then $\alpha$ is an isomorphism.

To prove this, we first restrict to suitable open subsets to obtain
a factorization of finite \'etale covers. The relation
\[
 x\sim_Q x'\quad\Longleftrightarrow\quad Q(x)=Q(x')
\]
on $T_m(y)$ is therefore monodromy-invariant.
Proposition~\ref{prop:sc-monodromy-partitions} gives an integer
$0\leq j\leq m$, independent of the general base point $y$, such that
\[
 Q(x)=Q(x')\quad\Longleftrightarrow\quad f^j(x)=f^j(x').
\]
Consequently, $Q$ and $f^j$ have the same general fibers and the
conclusion follows.

We now deduce Theorem~\ref{thm:main-semiconjugacy} from this
decomposition theorem.
By Stein factorization, we may suppose that $h$ is finite.
Write $F=f^m$, so that $F\circ h=h\circ u$.
The map $(h,u)$ takes $Y$ into the fiber product
$\Pbb^k\times_{\Pbb^k}Y$ formed using $F$ and $h$.
Normalize its image to obtain finite maps
\[
 p:Y\to Y_1,\qquad h_1:Y_1\to\Pbb^k,
 \qquad v:Y_1\to Y,
\]
with $h=h_1\circ p$ and $u=v\circ p$.

If $p$ is birational, a comparison of branch loci in the fiber
product shows that the branch locus $B_h$ of $h$ is totally invariant under $F$.
Simplicity excludes nonempty totally invariant hypersurfaces, so
$h$ is finite \'etale and hence an isomorphism.
Otherwise, setting $u_1=p\circ v$ gives
\[
 F\circ h_1=h_1\circ u_1,\qquad
 \deg_{\mathrm{top}}h_1<\deg_{\mathrm{top}}h.
\]
Induction on $\deg_{\mathrm{top}}h$, together with
Theorem~\ref{thm:sc-normal-factors}, gives $h=f^s\circ\mu$ and $u=\mu^{-1}\circ F\circ\mu$.

\medskip
The action on the rooted tree considered here is related to the
theory of iterated monodromy groups; see
Nekrashevych~\cite{Nekrashevych2005,Nekrashevych2011,Nekrashevych2014Models}
for the general theory and
Bartholdi--Nekrashevych~\cite{BartholdiNekrashevych2006,BartholdiNekrashevych2008}
for applications to PCF polynomials in one variable.
In the PCF setting, the branch loci of all iterates lie in a fixed
postcritical locus, so the finite-level actions fit together into
a self-similar action on the infinite preimage tree.

In parameter space, Gorbovickis--Taflin~\cite{GT24} use monodromy
to prove the irreducibility of spaces of endomorphisms with marked
periodic points and associated eigendirections. This irreducibility
is a key ingredient in their proof of the algebraic independence
of multiplier functions.

\subsection*{Organization}
Section~\ref{sec:simple-locus} establishes the generic ramification
properties and constructs the simple locus.
Section~\ref{sec:finite-level-monodromy} studies finite-level
monodromy and its invariant partitions.
Section~\ref{sec:semiconjugacy-rigidity} proves decomposition and
semiconjugacy rigidity, including Theorem~\ref{thm:main-semiconjugacy}.
Section~\ref{sect:proof} combines this with generic symmetry
rigidity to prove Theorem~\ref{thm:main} and treats regular polynomial
endomorphisms to prove Theorem~\ref{thm:regular-polynomial-centralizer}.
Finally, Section~\ref{sec:periodic-correspondences} classifies
periodic dominant correspondences and rational dynamical quotients
and proves Theorem~\ref{thm:periodic-correspondences}.

\subsection*{Convention}
We write compositions multiplicatively, so that $fg$ means $f\circ g$.

A property is said to hold for a general point, or generically,
on a variety if it holds on a dense Zariski open subset.

\subsection*{Acknowledgements}
I would like to thank Johan Taflin for sharing his insights on these
questions.

\section{Generic ramification and simple endomorphisms}
\label{sec:simple-locus}

The purpose of this section is to construct a dense Zariski open subset
$U_{\mathrm{simp}}\subset\End_d^k$, defined over $\Qbb$, whose complex
points are simple endomorphisms; see Proposition~\ref{prop:simple-open}.

Let $f:\Pbb^k\to \Pbb^k$ be an endomorphism of degree $d$. We write $\Rcal_f$
for the \emph{ramification divisor} (with multiplicities), so that
$K_{\Pbb^k}\sim f^*K_{\Pbb^k}+\Rcal_f$. Let $H$ be a hyperplane of $\Pbb^k$; then
$\Rcal_f\sim (k+1)(d-1)H$, where $\sim$ means linear equivalence. In particular, it has degree $(k+1)(d-1)$.
We write $R_f:=\Supp(\Rcal_f)$ for its support.

We write $\Bcal_f:=f_*\Rcal_f$ for the \emph{branch cycle} and
$B_f:=\Supp(\Bcal_f)=f(R_f)$ for its support, the \emph{branch
hypersurface}. By the projection formula,
\[
 \deg\Bcal_f
 =\Rcal_f\cdot(f^*H)^{k-1}
 =(k+1)(d-1)d^{k-1}=:M.
\]
Since $k\geq2$, put
$q:=M/d=(k+1)(d-1)d^{k-2}$. Unless otherwise specified, inverse
images of subvarieties are understood set-theoretically; pullbacks of
divisors with multiplicities are denoted by $f^*$.

The \emph{postcritical set} of $f$ is
\[
        \operatorname{PC}(f):=\bigcup_{m\geq1}f^m(R_f).
\]
We say that $f$ is \emph{postcritically finite} (PCF) if
$\operatorname{PC}(f)$ is Zariski closed.

\begin{defi}[\emph{Simple endomorphism}]\label{def:simple}
An endomorphism $f\in\End_d^k$ is called \emph{simple} if it satisfies
the following conditions:
\begin{enumerate}
\item[\textup{(S1)}] $\Rcal_f$ is integral;
\item[\textup{(S2)}] the morphism $f|_{R_f}:R_f\to B_f$ is birational;
\item[\textup{(S3)}] the branch hypersurface satisfies $B_f\subsetneq f^{-1}\bigl(f(B_f)\bigr)$;
\item[\textup{(S4)}] $f$ is not postcritically finite.
\end{enumerate}
\end{defi}

For $d\geq3$, the genericity of \textup{(S1)}, \textup{(S2)}, and
\textup{(S3)} follows from Ingram--Ramadas--Silverman's
arguments~\cite{pcfsilverman}. Here we use
Olechnowicz--Weinreich's theorem for \textup{(S1)} and give proofs
of \textup{(S2)} and \textup{(S3)} that also cover the quadratic case.
Ingram--Ramadas--Silverman also conjectured that \textup{(S4)} holds generically
\cite{pcfsilverman}; this was subsequently proved by
Gauthier--Taflin--Vigny~\cite{GauthierTaflinVignyPCF}.

\begin{lem}\label{lem:ramification-integral}
The locus
\[
 \Ucal_1:=\{f\in\End_d^k\mid\Rcal_f\text{ is integral}\}
\]
is dense Zariski open and defined over $\Qbb$.
\end{lem}

\begin{proof}
By \cite[Theorem~1.1]{OlechnowiczWeinreich2026}, the locus $\Ucal_1$
is nonempty and Zariski open, hence dense.
For every $\sigma\in\Aut(\Cbb/\Qbb)$, one has
$\Rcal_{f^\sigma}=(\Rcal_f)^\sigma$.
Thus $\Ucal_1$ is invariant under $\Aut(\Cbb/\Qbb)$, hence defined
over $\Qbb$.
\end{proof}

\begin{rem}
\label{rem:generic-ramification-literature}
For $d\geq3$, Ingram--Ramadas--Silverman~\cite[Theorem~14]{pcfsilverman}
proved that $\Rcal_f$ is of general type for a general
$f\in\End_d^k$.
Luo--Meng~\cite[Theorem~1.1(2)]{LuoMeng2026} showed that $\Rcal_f$
is singular for every $f$ when $k\geq4$. They also proved that the
pair $(\Pbb^k,\Rcal_f)$ is log canonical for a general $f$ in every
dimension~\cite[Theorem~1.1(3)]{LuoMeng2026}.

We can in fact show that $\dim\Sing(\Rcal_f)\leq k-4$ for a general
$f\in\End_d^k$.
\end{rem}

\begin{lem}
\label{lem:ramification-birational}
There exists a dense Zariski open subset
$\Ucal_2\subset \Ucal_1$, defined over $\Qbb$,
such that, for every $f\in \Ucal_2(\Cbb)$, the morphism
$f|_{R_f}:R_f\to B_f$ is birational.
\end{lem}

\begin{proof}
Consider the incidence
\[
\widetilde I_{\mathrm{bir}}
:=
\left\{
(f,p,q,y)\in
 \End_d^k\times
\bigl(\Pbb^k\times\Pbb^k\setminus\Delta\bigr)
\times\Pbb^k
\ \middle|\
f(p)=f(q)=y,\quad p,q\in R_f
\right\}.
\]
We first estimate the fibers of the projection
\[
\widetilde I_{\mathrm{bir}}
\longrightarrow
\bigl(\Pbb^k\times\Pbb^k\setminus\Delta\bigr)
\times\Pbb^k.
\]

Fix $(p,q,y)$. After linear changes of coordinates in the source
and target, we may assume $p=[1:0:\cdots:0]$,
$q=[0:1:0:\cdots:0]$, and $y=[1:0:\cdots:0]$. Write
$f=[F_0:\cdots:F_k]$. The conditions $f(p)=f(q)=y$ are
$F_i(p)=F_i(q)=0$ for $i=1,\ldots,k$.
They are $2k$ independent linear conditions on the coefficients of
$f$, since they independently fix the coefficients of $X_0^d$
and $X_1^d$ in $F_1,\ldots,F_k$. Moreover, $F_0(p)$ and
$F_0(q)$ are nonzero.

Suppose first that $d\geq3$. The first derivatives of $F_i/F_0$
at $p$ are controlled by the coefficients of $X_0^{d-1}X_j$, $1\leq j\leq k$,
whereas the first derivatives at $q$ are controlled by the
coefficients of $X_1^{d-1}X_j$, $j\in\{0,2,\ldots,k\}$.
For $d\geq3$, these two collections of monomials are disjoint.
They all vanish at $p$ and $q$, so their coefficients may be
varied freely without changing $f(p)=f(q)=y$. Consequently, the
two matrices $df_p$ and $df_q$ may be varied independently and
arbitrarily. The two determinantal conditions
$\det(df_p)=\det(df_q)=0$ therefore have codimension two.

It remains to consider $d=2$. For $1\leq i\leq k$, write the
part of $F_i$ which contributes to the first derivatives at $p$
or $q$ as
\[
 c_iX_0X_1
 +\sum_{j=2}^k a_{ij}X_0X_j
 +\sum_{j=2}^k b_{ij}X_1X_j.
\]
 Put $c=(c_1,\ldots,c_k)^{\mathsf t}$, $A=(a_{ij})$, and $B=(b_{ij})$,
where $A,B\in\operatorname{Mat}_{k\times(k-1)}(\Cbb)$. Since
$F_i(p)=F_i(q)=0$ for $i\geq1$, the quotient rule gives
\[
 df_p=F_0(p)^{-1}[\,c\mid A\,],
 \qquad
 df_q=F_0(q)^{-1}[\,c\mid B\,].
\]
On the locus $c\neq0$, the polynomial $\det[\,c\mid A\,]$ is
a nonzero equation in the variables $A$, while
$\det[\,c\mid B\,]$ is a nonzero equation in the disjoint set of
variables $B$. Their simultaneous vanishing therefore has
codimension two. On the locus $c=0$, both determinants vanish
automatically, but $c=0$ itself has codimension $k\geq2$.
Thus, also when $d=2$, the simultaneous conditions
$p,q\in R_f$ have codimension at least two after
$f(p)=f(q)=y$ has been imposed.

It follows that every fiber of the displayed projection has
codimension at least $2k+2$ in $\End_d^k$. Since the parameter space of
$(p,q,y)$ has dimension $3k$, we obtain
\[
 \dim\widetilde I_{\mathrm{bir}}
 \leq \dim\End_d^k+3k-(2k+2)
 =\dim\End_d^k+k-2.
\]
Forgetting $y$ identifies $\widetilde I_{\mathrm{bir}}$ with
\[
 I_{\mathrm{bir}}
 :=
 \left\{
 (f,p,q)\in
 \End_d^k\times\bigl(\Pbb^k\times\Pbb^k\setminus\Delta\bigr)
 \ \middle|\
 p,q\in R_f,\quad f(p)=f(q)
 \right\},
\]
since $y$ is uniquely determined by $f,p,q$. Hence
$\dim I_{\mathrm{bir}}\leq\dim\End_d^k+k-2$.

Let
\[
 E:=
 \left\{
 f\in\End_d^k
 \ \middle|\
 \dim(I_{\mathrm{bir}})_f=k-1
 \right\}.
\]
By the constructibility of fiber dimension for morphisms of finite
type, $E$ is constructible. Since the projection
$I_{\mathrm{bir}}\to\End_d^k$ is defined over $\Qbb$, so is $E$. Moreover,
\[
 \dim E+(k-1)
 \leq\dim I_{\mathrm{bir}}
 \leq\dim\End_d^k+k-2,
\]
so $\dim E\leq\dim\End_d^k-1$. Thus
$\overline E\subsetneq\End_d^k$ is a
proper closed subset defined over $\Qbb$. Then
$\Ucal_2:=\Ucal_1\cap
\bigl(\End_d^k\setminus\overline E\bigr)$ has the required properties.
\end{proof}

\begin{rem}
\label{rem:birationality-irs}
For $d\geq3$, the generic birationality of $f|_{R_f}$ also follows
from~\cite[Propositions~21 and~23(a)]{pcfsilverman}, taking $\ell=1$.
\end{rem}

\begin{lem}
\label{lem:generic-branch-support}
There exists a dense Zariski open subset
$\Ucal_3\subset \Ucal_2$, defined over $\Qbb$, such
that every $f\in \Ucal_3(\Cbb)$ satisfies \textup{(S3)}.
\end{lem}

\begin{proof}
Put $W_N:=H^0(\Pbb^k,\Ocal(N))$, and choose a nonzero equation
$\Phi_f\in W_M$ of the branch cycle $\Bcal_f$. This choice is unique
up to a nonzero scalar, and the associated projective class defines the
$\Qbb$-morphism
\[
 \beta:\End_d^k\longrightarrow\Pbb(W_M),
 \qquad f\longmapsto[\Phi_f].
\]
If
$f=[F_0:\cdots:F_k]$, substitution defines a morphism over $\Qbb$
\[
 \begin{aligned}
 \gamma:\End_d^k\times\Pbb(W_q)&\longrightarrow\Pbb(W_M),\\
 (f,[C])&\longmapsto[C(F_0,\ldots,F_k)].
 \end{aligned}
\]
The
incidence subset
\[
 \Ical:=\left\{(f,[C])\mid
 [\Phi_f]=[C(F_0,\ldots,F_k)]\right\}
 \subset\End_d^k\times\Pbb(W_q)
\]
is closed and defined over $\Qbb$. Its image $Z_{\mathrm{pb}}$ in $\End_d^k$ is closed and defined over
$\Qbb$.

We show that $Z_{\mathrm{pb}}$ is proper. For
$L\in\PGL_{k+1}$, let
\[
 p_d=[X_0^d:\cdots:X_k^d],\qquad
 f_L:=L p_d,\qquad
 \Delta_d:=\mu_d^{k+1}/\mu_d\subset\PGL_{k+1}.
\]
Every form in the image of $f_L^*$ is invariant under
$\Delta_d$. If $H_i:=\{X_i=0\}$, then
\[
 \Bcal_{f_L}
 =(d-1)d^{k-1}\sum_{i=0}^k[L(H_i)].
\]
For a general $L\in\PGL_{k+1}$, the hyperplane arrangement $\mathscr H_L:=\bigcup_{i=0}^k L(H_i)$
is invariant under no nontrivial element of $\Delta_d$. To see this,
put
\[
 \mathcal N:=
 \left\{
 \left[\operatorname{diag}(\lambda_0,\ldots,\lambda_k)P_\sigma\right]
 \in\PGL_{k+1}(\Cbb)
 \ \middle|\
 \lambda_i\in\Cbb^\times,\ \sigma\in S_{k+1}
 \right\},
\]
where $P_\sigma$ is the permutation matrix associated with $\sigma$.
This is precisely the stabilizer of the coordinate arrangement
$\bigcup_iH_i$. Equivalently, $\mathcal N$ is the stabilizer of
\[
 [X_0\cdots X_k]\in
 \Pbb\bigl(H^0(\Pbb^k,\Ocal(k+1))\bigr)
\]
under the natural action of $\PGL_{k+1}$. Hence $\mathcal N$ is closed.
For
$\delta\in\Delta_d$, one has
\[
   \delta(\mathscr H_L)=\mathscr H_L
   \quad\Longleftrightarrow\quad
   L^{-1}\delta L\in\mathcal N.
\]

Fix $1\neq\delta\in\Delta_d$ and consider the Zariski closed set
\[
   Z_\delta:=
   \left\{
      L\in\PGL_{k+1}\mid L^{-1}\delta L\in\mathcal N
   \right\}.
\]
We claim that it is proper.
Choose a diagonal
representative $D=\operatorname{diag}(\lambda_0,\ldots,\lambda_k)$
of $\delta$. Since $\delta\neq1$ in $\PGL_{k+1}$, there exist $a\neq b$ with
$\lambda_a\neq\lambda_b$. For $t\in\Cbb$, put
$L_t=I+tE_{ab}$. Since $E_{ab}^2=0$, we have
$L_t^{-1}=I-tE_{ab}$ and therefore
\[
\begin{aligned}
   L_t^{-1}DL_t
   =(I-tE_{ab})D(I+tE_{ab}) 
   =D+t(\lambda_a-\lambda_b)E_{ab}.
\end{aligned}
\]
For $t\neq0$, this matrix has all its diagonal entries nonzero and also
has a nonzero off-diagonal $(a,b)$-entry. It is therefore not
monomial, and hence $L_t^{-1}\delta L_t\notin\mathcal N$. Thus
$L_t\notin Z_\delta$, proving that $Z_\delta$ is proper.

The finite union $\bigcup_{1\neq\delta\in\Delta_d} Z_\delta$
is a proper closed subset of $\PGL_{k+1}$. Every $L$ outside this
union has the required property: no nontrivial element of $\Delta_d$
preserves $\mathscr H_L$.

Consequently, $[\Phi_{f_L}]$ is not invariant under $\Delta_d$ for
general $L$, whereas every point of
$\Pbb(\operatorname{im}(f_L^*))$ is invariant under $\Delta_d$.
Thus $Z_{\mathrm{pb}}$ is proper. Set
$\Ucal_3:=\Ucal_2\setminus Z_{\mathrm{pb}}$; this is a dense Zariski open
subset defined over $\Qbb$.

Let $f\in \Ucal_3(\Cbb)$.
Suppose that \textup{(S3)} fails, and put $B_1:=f(B_f)$. Then
$f^{-1}(B_1)=B_f$. Since $B_f$ and $B_1$ are integral, $f^*B_1=eB_f$
for some integer $e\geq1$. If $e\geq2$, then $f$ is ramified at the
general point of $B_f$, so $B_f=R_f$, contradicting
\[
 \deg R_f=(k+1)(d-1)
 \quad\text{and}\quad
 \deg B_f=(k+1)(d-1)d^{k-1}.
\]
Thus $e=1$. Taking degrees gives $\deg B_1=q$. If $C\in W_q$ is an
equation of $B_1$, then
\[
 \operatorname{div}(f^*C)=f^*B_1=B_f=\Bcal_f
 =\operatorname{div}(\Phi_f),
\]
and hence $[\Phi_f]=[f^*C]$, contrary to
$f\notin Z_{\mathrm{pb}}$.
\end{proof}

\begin{rem}
\label{rem:branch-support-irs}
For $k\geq3$ and $d\geq3$, the generic validity of \textup{(S3)} can also be deduced
from the argument of~\cite[Lemma~29 and Proposition~23(a)]{pcfsilverman},
with $\ell=2$, together with \textup{(S2)}.
\end{rem}

\begin{thm}[Gauthier--Taflin--Vigny~{\cite[Theorem~B]{GauthierTaflinVignyPCF}}]
\label{thm:pcf-sparsity}
There exists a dense Zariski open subset
$\Ucal_4\subset\End_d^k$, defined over $\Qbb$, such that
no $f\in \Ucal_4(\Cbb)$ is PCF.
\end{thm}

The PCF locus is invariant under $\Aut(\Cbb/\Qbb)$; hence its Zariski
closure, and therefore its open complement, is defined over $\Qbb$.

Combining the preceding results, set
\[
U_{\mathrm{simp}}:=\bigcap_{i=1}^4\Ucal_i.
\]

\begin{prop}\label{prop:simple-open}
The subset $U_{\mathrm{simp}}\subset\End_d^k$ is dense Zariski open
and defined over $\Qbb$. Every
$f\in U_{\mathrm{simp}}(\Cbb)$ is simple.
\end{prop}

The following consequence of simplicity will be used to exclude
contracted hypersurfaces in rational semiconjugacies.

\begin{cor}
\label{cor:no-totally-invariant-hypersurface}
Let $f\in\End_d^k$ be simple. For every $m\geq1$, there is no
nonempty hypersurface $D\subset\Pbb^k$ satisfying
\[
 (f^m)^{-1}(D)=D.
\]
\end{cor}

\begin{proof}
Suppose that such a reduced hypersurface $D$ exists. After replacing
$m$ by a positive multiple, we may assume $D$ is irreducible.
Since $(f^m)^*D=d^mD$ and $d^m>1$, we have
$D\subset R_{f^m}=\bigcup_{j=0}^{m-1}f^{-j}(R_f)$.
So $f^j(D)=R_f$ for some
$0\leq j<m$. Using $f^m(D)=D$, we obtain
$f^m(R_f)=f^{m+j}(D)=f^j(f^m(D))=f^j(D)=R_f$.
Therefore $\operatorname{PC}(f)=\bigcup_{r\geq1}f^r(R_f)
=\bigcup_{r=1}^{m}f^r(R_f)$ is Zariski closed. This says that $f$
is PCF, contradicting \textup{(S4)}.
\end{proof}

\begin{rem}
For a general endomorphism, the conclusion of
Corollary~\ref{cor:no-totally-invariant-hypersurface} follows from
Dinh--Sibony~\cite[Lemma~1.69]{DSbook}.
Here we need this conclusion for every simple endomorphism.
\end{rem}

\section{Finite-level monodromy and tree actions}
\label{sec:finite-level-monodromy}
We study the finite-level monodromy action on the rooted preimage
tree.
Throughout this section and Section~\ref{sec:semiconjugacy-rigidity},
let $f\in\End_d^k$ be simple.

For $j\geq0$, put
\[
        \Rcal_j:=(f^j)^*\Rcal_f,
        \qquad
        R_j:=\Supp(\Rcal_j)=f^{-j}(R_f),
        \qquad
        B_j:=f^j(B_f).
\]
If $g$ is another surjective endomorphism of $\Pbb^k$, then the chain
rule gives $\Rcal_{fg}=\Rcal_g+g^*\Rcal_f$, and hence
$B_{fg}=B_f\cup f(B_g)$. In particular, by induction,
$\Rcal_{f^m}=\sum_{j=0}^{m-1}\Rcal_j$ and
$B_{f^m}=\bigcup_{j=0}^{m-1}B_j$. 

With this notation, for $s\geq0$ put
\[
 D_0:=\varnothing,
 \qquad
 D_s:=B_{f^s}=B_0\cup\cdots\cup B_{s-1}\quad(s\geq1),
 \qquad
 U_s:=\Pbb^k\setminus D_s.
\]
Fix an integer $m\geq1$ and a base point $y\in U_m$.

\begin{defi}[\emph{Rooted preimage tree}]\label{def:tree}
The \emph{root} is $y$. For $0\leq s\leq m$, the
\emph{vertices of depth $s$} are
$T_s(y):=f^{-s}(y)$.
Vertices at different depths are regarded as distinct, even when
they represent the same point of $\Pbb^k$.
The \emph{parent} of $x\in T_s(y)$, $s\geq1$, is
$f(x)\in T_{s-1}(y)$.
If $v\in T_s(y)$, its \emph{descendants at depth $m$} form the set
$\Desc_m(v):=(f^{m-s})^{-1}(v)=T_{m-s}(v)\subset T_m(y)$.
\end{defi}

For $z\in U_s$, let
$\rho_{s,z}:\pi_1(U_s,z)\to\Sym(T_s(z))$
be the \emph{monodromy representation} of the \'etale cover
$f^s:f^{-s}(U_s)\to U_s$, and let
$G_s(z):=\im(\rho_{s,z})$ be the corresponding \emph{monodromy group}.
For the fixed base point $y$, we abbreviate $\rho_s:=\rho_{s,y}$.
Note that $T_0(z)=\{z\}$ and $G_0(z)=\{1\}$.

\subsection{Meridians}
Let $M$ be a connected complex manifold, let $H\subset M$ be a
hypersurface, and let $D$ be an irreducible component of $H$. Put
$M^\circ:=M\setminus H$. Let $D^{\mathrm{sm}}$ be the smooth locus of $D$. At a point
$p\in D^{\mathrm{sm}}\setminus\bigcup_{D'\neq D}D'$, a
\emph{transverse disk to $D$ at $p$} is a holomorphic map
$\iota:\Delta_2:=\{z\in\Cbb\mid |z|<2\}\to M$ such that
\[
        \iota(0)=p,
        \qquad
        \iota^{-1}(H)=\{0\},
        \qquad
        d\iota_0(T_0\Delta_2)\not\subset T_pD.
\]
The counterclockwise circle $\iota|_{S^1}$ is the
\emph{positive transverse circle at $p$}.

\begin{defi}[\emph{Meridian}]\label{def:meridian}
Fix $q\in M^\circ$, choose a path $c$ in $M^\circ$ from $q$ to
$\iota(1)$, and let $\delta:=\iota|_{S^1}$. The based loop
$c\delta c^{-1}$ is called a \emph{meridian} of $D$ at $p$.
Changing $c$ conjugates the corresponding element of
$\pi_1(M^\circ,q)$.
\end{defi}

For any covering under consideration, the connecting path $c$
identifies its fiber over $q$ with the local sheets over $\iota(1)$ by
analytic continuation; we suppress this identification from the
notation. Replacing $c$ conjugates the meridian and changes this
identification by monodromy.

The smooth locus of an irreducible complex hypersurface, after removing
its intersections with the other components, is path connected. Hence
the meridians of one irreducible component form a single conjugacy class.
If $H\subset H'$ are hypersurfaces in $M$, then the inclusion
$M\setminus H'\hookrightarrow M\setminus H$ induces a surjection on
fundamental groups. Moreover, if
$H=D_1\cup\cdots\cup D_r$, then the kernel of $\pi_1(M^\circ,q)\to \pi_1(M,q)$
is the normal subgroup generated by meridians around the components $D_i$. In
particular, since $\Pbb^k$ is simply connected, $\pi_1(U_m,y)$ is normally generated
by meridians around $B_0,\ldots,B_{m-1}$. See
\cite[Definition~1.33 and Propositions~1.34 and~2.2]{Cogolludo2011}
or \cite[Section 2.1]{AguilarAguilar} for these standard facts.

\begin{lem}\label{lem:generic-fold}
At a general point $p\in R_f$, there are analytic coordinates centered
at $p$ and $f(p)$ in which
\[
        f(x_1,\ldots,x_k)=(x_1,\ldots,x_{k-1},x_k^2).
\]
In particular, the local monodromy around $B_f$ is a transposition.
\end{lem}

We say that $f$ is a \emph{(simple) fold} at $p$.

\begin{proof}
By \textup{(S1)}, $\Rcal_f$ is reduced, and by \textup{(S2)},
$f|_{R_f}:R_f\to B_f$ is birational. Let $p\in R_f$ be a general
point and put $q=f(p)$. We may choose $p$ so that $R_f$ is smooth at
$p$, $B_f$ is smooth at $q$, and $f|_{R_f}:R_f\to B_f$ is an
isomorphism in a neighborhood of $p$ and $q$.
Choose analytic local coordinates $(y_1,\ldots,y_k)$ at $q$ such that
$B_f=\{y_k=0\}$. Since $f|_{R_f}$ is locally an isomorphism at $p$,
the functions $f^*y_1,\ldots,f^*y_{k-1}$ restrict to a system of
local coordinates on $R_f$ near $p$. Let $x_k$ be a local equation of
$R_f$ at $p$, and set $x_i=f^*y_i$, $i=1,\ldots,k-1$. Then
$(x_1,\ldots,x_{k-1},x_k)$ is a system of analytic local coordinates
at $p$.

Since $p$ is general, we may also assume that $R_f$ is the only
irreducible component of $f^{-1}(B_f)$ passing through $p$. Since
$y_k=0$ is a local equation for $B_f$, we have
$f^*y_k=u x_k^e$ for some analytic unit $u$ and some integer $e\geq2$.
The multiplicity of the ramification divisor along $R_f$ is $e-1$.
Since $\Rcal_f$ is reduced, $e=2$. After shrinking the neighborhood,
$u$ admits a square root. Replacing $x_k$ by $\sqrt{u}\,x_k$, we get
$f^*y_k=x_k^2$, which is the desired local normal form.
By \textup{(S2)}, a general point of $B_f$ has exactly one ramified
preimage, so the monodromy of a meridian is a single transposition.
\end{proof}

\begin{lem}\label{lem:G1-full}
The monodromy group $G_1(y)$ is the full symmetric group $S_{d^k}$.
\end{lem}

\begin{proof}
The covering space $f^{-1}(U_1)$ is a dense
Zariski open subset of the irreducible variety $\Pbb^k$, so it is
connected. Thus $G_1(y)\subset S_{d^k}$ acts transitively.
By Lemma~\ref{lem:generic-fold}, the monodromies of meridians around
$B_f$ are transpositions. Since these meridians generate
$\pi_1(U_1,y)$, the transitive group $G_1(y)$ is generated by
transpositions, hence $G_1(y)=S_{d^k}$.
\end{proof}

For $0\leq r\leq m$, the inclusion $i:U_m\hookrightarrow U_r$ induces
a surjection $i_*:\pi_1(U_m,y)\twoheadrightarrow\pi_1(U_r,y)$ by the
surjectivity recalled above. Path lifting gives, for
every $\gamma\in\pi_1(U_m,y)$ and $x\in T_m(y)$,
\begin{equation}
 f^{m-r}\bigl(\rho_m(\gamma)(x)\bigr)
 =\rho_r(i_*\gamma)\bigl(f^{m-r}(x)\bigr).
 \label{eq:tree-equivariance}
\end{equation}
Because $f^{m-r}:T_m(y)\twoheadrightarrow T_r(y)$ is surjective,
Equation~\eqref{eq:tree-equivariance} shows that
$\ker(\rho_m)\subseteq\ker(\rho_r\circ i_*)$.
Hence the formula
\begin{equation}
 \theta_{m,r}:G_m(y)\longrightarrow G_r(y),
 \qquad
 \theta_{m,r}\bigl(\rho_m(\gamma)\bigr):=\rho_r(i_*\gamma)
 \label{eq:theta-mr}
\end{equation}
is well defined. It is surjective because $i_*$ is surjective. We use
$\theta_{m,r}$ to let $G_m(y)$ act on $T_r(y)$, and write
$(G_m(y))_v:=\{\sigma\in G_m(y)\mid\sigma(v)=v\}$ for the stabilizer of a
vertex $v\in T_r(y)$.

\begin{lem}\label{lem:stabilizer-action}
Let $0\leq r\leq m$ and $v\in T_r(y)$. The group $G_m(y)$ acts
transitively on $T_r(y)$, and the image of $(G_m(y))_v$ on
$\Desc_m(v)=T_{m-r}(v)$ is precisely $G_{m-r}(v)$. 
\end{lem}

\begin{proof}
The covering space $f^{-r}(U_r)$ is a dense Zariski open subset of
$\Pbb^k$, hence is connected. Thus $G_r(y)$ acts transitively on $T_r(y)$,
and the surjectivity of $\theta_{m,r}$ gives the first assertion.

Next, note that $v\in U_{m-r}$. Indeed, if $v\in B_s$ for some
$0\leq s<m-r$, then $y=f^r(v)\in f^r(B_s)=B_{s+r}\subset D_m$,
contrary to $y\in U_m$. Put $W:=f^{-r}(U_m)$. A loop
$\gamma\in\pi_1(U_m,y)$ fixes $v$ if and only if its lift by
$f^r:W\to U_m$ starting at $v$ is a loop
$\widetilde\gamma\in\pi_1(W,v)$. For $x\in\Desc_m(v)$, lifting
$\gamma$ by $f^m$ is the same as first lifting it to
$\widetilde\gamma$ and then lifting $\widetilde\gamma$ by $f^{m-r}$.
It follows that the image of $(G_m(y))_v$ on $\Desc_m(v)$ is the monodromy
image of $f^{m-r}:f^{-(m-r)}(W)\to W$.

For $0\leq s<m-r$, the inclusion
$f^r(B_s)=B_{s+r}\subset D_m$ gives
$W=f^{-r}(U_m)\subset U_{m-r}$.
The relative complement $U_{m-r}\setminus W$ is a hypersurface (or is
empty), so the same surjectivity argument makes
$\pi_1(W,v)\to\pi_1(U_{m-r},v)$ surjective. The preceding monodromy
image is therefore exactly $G_{m-r}(v)$, proving the second assertion.
\end{proof}

\subsection{\texorpdfstring{Level-$j$ fold pairs}{Level-j fold pairs}}

Before defining fold pairs, we record two consequences of simplicity
that control the ramification and branch loci at different levels.

\begin{lem}\label{lem:ram-level-separation}
The irreducible hypersurfaces $B_i$, $i\geq0$, together with all
irreducible components of the hypersurfaces $R_j$, $j\geq0$, are
pairwise distinct.
\end{lem}

\begin{proof}
Each $B_i=f^{i+1}(R_f)$ is irreducible because $R_f$ is integral and
$f$ is finite.
If $B_i=B_j$ for some $j>i$, then
$f^{j+1}(R_f)=f^{i+1}(R_f)$, so $R_f$ is preperiodic and $f$ is PCF,
contradicting \textup{(S4)}.
Thus the $B_i$ are pairwise distinct.

Assume that an irreducible hypersurface $C$ is a common component of
$R_i$ and $R_j$ with $i<j$. Since $R_f$ is integral by \textup{(S1)},
$f^i(C)=f^j(C)=R_f$. Hence $f^{j-i}(R_f)=R_f$, again contradicting
\textup{(S4)}.

It remains to separate the two families. Let $C$ be an irreducible
component of $R_j$. Since $f^j$ is finite and $R_f$ is integral, one
has $f^j(C)=R_f$. If $C=B_i$ for some $i\geq0$, then
\[
        R_f=f^j(C)=f^j(B_i)=B_{i+j}.
\]
Applying $f$ gives $B_0=f(R_f)=B_{i+j+1}$, contradicting the pairwise
distinctness of the $B_i$ because $i+j+1>0$.
\end{proof}

\begin{lem}\label{lem:reduced-ramification-levels}
The divisor $\Rcal_j=(f^j)^*\Rcal_f$ is reduced for every $j\geq0$.
\end{lem}

\begin{proof}
The case $j=0$ follows from \textup{(S1)}. Suppose that $j\geq1$ and
that $\Rcal_j$ is nonreduced along an irreducible component $C$. At
the generic points of $C$ and $R_f$, let $s=0$ and $t=0$ be local
equations. Since $\Rcal_f$ is reduced, $t\circ f^j=u s^e$ for a unit
$u$, and the coefficient of $C$ in $\Rcal_j$ is $e$; nonreducedness
gives $e>1$, so $f^j$ is ramified along $C$. Hence $C$ is contained
in the ramification divisor $\Rcal_{f^j}$, whose support is contained
in $R_0\cup\cdots\cup R_{j-1}$. Thus $C$ is also an irreducible
component of some $R_s$ with $s<j$, contradicting
Lemma~\ref{lem:ram-level-separation}.
\end{proof}

Fix $0\leq j<m$, and put $r:=m-j-1$. Set $E_{m,j}:=f^{m}(R_j)=B_r$.
We decompose $f^m=f^r f f^j$. There is a dense
Zariski open subset $E_{m,j}^{\circ}\subset E_{m,j}$ such that, for
every $a\in E_{m,j}^{\circ}$ and every $x\in R_j\cap f^{-m}(a)$, one
has
\begin{equation}\label{eq:etale-fold-etale}
x
\xrightarrow[\text{\'etale}]{f^j}
p:=f^j(x)\in R_f
\xrightarrow[\text{fold}]{f}
b:=f^{j+1}(x)\in B_f
\xrightarrow[\text{\'etale}]{f^r}
a\in E_{m,j}.
\end{equation}
To obtain $E_{m,j}^{\circ}$, delete from $B_r$ its singular locus and
its intersections with the other $B_i$, $0\leq i<m$ and $i\neq r$,
together with the
$f^m$-images of the loci in $R_j$ where $f^j$ is ramified, where $f$
is not a fold at $f^j(x)$, or where $f^r$ is ramified at
$f^{j+1}(x)$. Lemmas~\ref{lem:ram-level-separation} and
\ref{lem:generic-fold} show that the resulting open subset is dense.
Thus the first and last maps in \eqref{eq:etale-fold-etale} are locally
\'etale and the middle map is a fold. In particular, $f^m$ is
a fold at every such $x$.
Consequently, the monodromy of a meridian at $a$ is a product of
disjoint transpositions, one for each such fold point; see
Figure~\ref{fig:meridian-j1-m3}.

\begin{defi}[\emph{Level-$j$ fold pair}]\label{def:level-pair}
A \emph{level-$j$ fold pair at depth $m$} is an unordered pair
$\{x,x'\}\subset T_m(y)$ such that $(x\ x')$ is a transposition factor
in $\rho_m(\mu)$ for some meridian $\mu$ at a point of
$E_{m,j}^\circ$. We denote the set of all such pairs by
$\Fcal_{m,j}(y)$.

If $j=m-1$, so that $E_{m,j}=B_0$, the pair is called a
\emph{top-level fold pair}, and the image $\rho_m(\mu)\in G_m(y)$ of any
such meridian is called a \emph{top-meridian element at depth $m$}.
\end{defi}

Since meridians around $B_{m-j-1}$ are conjugate, shrinking
$E^\circ_{m,j}$ to a smaller dense Zariski open subset does not
change $\Fcal_{m,j}(y)$. Moreover, $\Fcal_{m,j}(y)$ is
$G_m(y)$-invariant.

Let $e=\{x,x'\}\in\Fcal_{m,j}(y)$. Transporting the local description
\eqref{eq:etale-fold-etale} to $y$, we obtain
\[
        f^{j+1}(x)=f^{j+1}(x'),
        \qquad
        f^s(x)\neq f^s(x')\quad(0\leq s\leq j).
\]

\subsection{Orbits and invariant equivalence relations}

The aim of this subsection is to prove the following proposition.

\begin{prop}
\label{prop:sc-monodromy-partitions}
Let $m\geq1$ and $y\in U_m$. For every
$G_m(y)$-invariant equivalence relation $\sim_R$
on $T_m(y)$, there is a unique $0\leq j\leq m$ such that
\[
 x\sim_R x'
 \quad\Longleftrightarrow\quad
 f^j(x)=f^j(x')
 \qquad(x,x'\in T_m(y)).
\]
\end{prop}

\begin{defi}[\emph{Orbit graph $\Gamma_m(e;y)$}]
\label{def:orbit-graph}
For an unordered pair $e=\{a,b\}$ of distinct points in $T_m(y)$, define the
\emph{orbit graph} $\Gamma_m(e;y)$ to be the graph with vertex set
$T_m(y)$ and edge set
$G_m(y)\cdot e=\{\sigma(e)\mid\sigma\in G_m(y)\}$.
\end{defi}

We first analyze the kernel of the action on $T_1(y)$.
Assume $m\geq2$, fix $y\in U_m$, and put
\[
        K_m:=\ker(\theta_{m,1}:G_m(y)\to G_1(y)).
\]
The group $K_m$ acts trivially on $T_1(y)$ and hence preserves
$\Desc_m(v)$ for every $v\in T_1(y)$. For $v\in T_1(y)$, let
$L_v:=K_m|_{\Desc_m(v)}$ be the permutation group induced by $K_m$ on
$\Desc_m(v)$.

\begin{lem}\label{lem:kernel}
For $m\geq2$ the following assertions hold.
\begin{enumerate}
\item For every $v\in T_1(y)$,
      \[
             L_v=G_{m-1}(v)
      \]
      as permutation groups on
      $\Desc_m(v)=T_{m-1}(v)$.
\item There exist $\eta\in K_m$ and a nonempty proper subset
      $S\subsetneq T_1(y)$ such that
      \begin{enumerate}
      \item $\eta$ acts trivially on $\Desc_m(v)$ for every
            $v\in T_1(y)\setminus S$;
      \item $\eta$ acts on $\Desc_m(v)$ as a top-meridian element of
            the lower cover $f^{m-1}$ for every $v\in S$.
      \end{enumerate}
\end{enumerate}
\end{lem}

\begin{proof}
(1) Fix $v\in T_1(y)$ and $0\leq a\leq m-2$. Choose a general point
$p\in B_{a+1}$ lying on no other component of $D_m$, and choose
$z\in B_a$ with $f(z)=p$. Since $p\notin B_0$, the map $f$ is
\'etale at $z$. A transverse disk to $B_{a+1}$ at $p$ meeting no other component of $D_m$ therefore
lifts near $z$ to a transverse disk to $B_a$.
Put $W:=f^{-1}(U_m)\subset U_{m-1}$. Let
$\widetilde\mu_a\in\pi_1(W,v)$ be a meridian of $B_a$ associated with
the lifted transverse disk, and put
$\mu_{a+1}:=f_*(\widetilde\mu_a)\in\pi_1(U_m,y)$.
 Its image in
$\pi_1(U_1,y)$ is trivial since the transverse disk at $p$ is
disjoint from $B_0$. Hence $\rho_m(\mu_{a+1})\in K_m$, and its
restriction to $\Desc_m(v)$ is the lower monodromy of a meridian of
$B_a$. These meridians normally generate $G_{m-1}(v)$. Moreover,
$K_m\triangleleft G_m(y)$, and Lemma~\ref{lem:stabilizer-action} identifies
the image of $(G_m(y))_v$ on $\Desc_m(v)$ with $G_{m-1}(v)$. Hence
$L_v\triangleleft G_{m-1}(v)$, and therefore $L_v=G_{m-1}(v)$,
proving (1).

(2) Choose a general point $p\in B_1$, a small transverse disk
$\Delta$ at $p$ meeting no other component of $D_m$, and let
$\mu_1\in\pi_1(U_m,y)$ be the resulting meridian based at $y$.
For each $v\in T_1(y)$, let $z_v\in f^{-1}(p)$ be the value at $p$
of the local inverse branch of $f$ corresponding to $v$, and define
$S:=\{v\in T_1(y)\mid z_v\in B_0\}$.
The set $S$ is nonempty because $f(B_0)=B_1$. It is proper by
condition \textup{(S3)}: there is an irreducible component of
$f^{-1}(B_1)$ distinct from $B_0$. This component dominates $B_1$
and therefore meets the fiber over a general $p$ outside $B_0$.

Put $\eta:=\rho_m(\mu_1)$. Since $\Delta$ is disjoint from $B_0$,
the image of $\mu_1$ in $\pi_1(U_1,y)$ is trivial, and hence
$\eta\in K_m$. If $v\in S$, the
lift of $\mu_1$ from $v$ is a meridian of $B_0$, so its action on
$\Desc_m(v)$ is a lower top-meridian element. If $v\notin S$, then
$z_v\notin D_{m-1}$. Indeed, $z_v\notin B_0$, and if
$z_v\in B_a$ for some $1\leq a\leq m-2$, then
$p=f(z_v)\in B_{a+1}$, contrary to the choice of $p$. The lifted
small circle is therefore null-homotopic in $U_{m-1}$, so $\eta$ acts
trivially on $\Desc_m(v)$. This proves (2).
\end{proof}

\begin{lem}
\label{lem:two-coordinate-extraction}
Let $r\geq2$, $y\in U_r$, and let
$v,w\in T_1(y)$ be distinct. The image $P$ of
\[
        K_r\longrightarrow L_v\times L_w
\]
contains a product $H_v\times H_w$, where
$H_v\triangleleft L_v$ and $H_w\triangleleft L_w$ are transitive
normal subgroups.
\end{lem}

\begin{proof}
For $r\geq1$, the following assertions are understood to hold for
every base point $y\in U_r$.
\begin{description}
\item[$\Tfrak_r$] For every top-meridian element $\tau\in G_r(y)$,
the normal closure
$\langle\!\langle\tau\rangle\!\rangle_{G_r(y)}$ acts transitively on
$T_r(y)$.
\item[$\Cfrak_r$] For every top-level fold pair
$e\in\Fcal_{r,r-1}(y)$, the orbit graph $\Gamma_r(e;y)$ is connected.
\end{description}

We first show that, if $\Tfrak_{r-1}$ holds, then the conclusion
of the lemma holds for $r$.
Choose $\eta\in K_r$ and
$\varnothing\neq S\subsetneq T_1(y)$ as in
Lemma~\ref{lem:kernel}(2), and choose
$s\in S$ and $t\in T_1(y)\setminus S$. Since
$G_1(y)=S_{d^k}$ and $\theta_{r,1}$ is surjective, there is
$\sigma\in G_r(y)$ such that $\sigma(s)=v$ and $\sigma(t)=w$.
Normality of $K_r$ gives $\eta_v:=\sigma\eta\sigma^{-1}\in K_r$.
Its image in $L_v\times L_w$ has the form $(a_v,1)$, where $a_v$ is
a lower top-meridian element. Put
$H_v:=\langle\!\langle a_v\rangle\!\rangle_{L_v}$.
By Lemma~\ref{lem:kernel}(1), $L_v=G_{r-1}(v)$, and the assertion
$\Tfrak_{r-1}$ shows that $H_v$ is transitive on $\Desc_r(v)$.

The projection $P\to L_v$ is surjective. Hence, for every
$\ell\in L_v$, choose $b\in L_w$ with $(\ell,b)\in P$. Then
$(\ell,b)(a_v,1)(\ell,b)^{-1}
=(\ell a_v\ell^{-1},1)\in P$.
These elements generate $H_v\times1$, so $H_v\times1\subset P$.
Choosing instead an element of $G_r(y)$ sending $s$ to $w$ and $t$
to $v$ gives, by the same argument, a transitive normal subgroup
$H_w\triangleleft L_w$ with $1\times H_w\subset P$. Therefore
$H_v\times H_w\subset P$.

We prove the two connectivity assertions simultaneously.

\emph{Step 1: $\Tfrak_1$ and $\Cfrak_1$ hold.}
This follows from $G_1(y)=S_{d^k}$. A meridian around $B_0$ acts as
a transposition, whose normal closure is the full symmetric group,
and the orbit of one fold pair is the complete graph on $T_1(y)$.

\emph{Step 2: $\Tfrak_{r-1}$ implies $\Cfrak_r$.}
Let $r\geq2$ and fix
$e=\{x,x'\}\in\Fcal_{r,r-1}(y)$. Put $v:=f^{r-1}(x)$ and
$v':=f^{r-1}(x')$.
Equation~\eqref{eq:etale-fold-etale} gives $v\neq v'$. Since $G_1(y)$ is the full
symmetric group and $\theta_{r,1}$ is surjective, the $G_r(y)$-orbit
of $e$ contains an edge
$e_{w,w'}=\{z,z'\}$ below every unordered pair of distinct vertices
$w,w'\in T_1(y)$. Given arbitrary
$\alpha\in\Desc_r(w)$ and $\beta\in\Desc_r(w')$,
the preceding argument supplies
$\kappa\in K_r$ with $\kappa(z)=\alpha$ and $\kappa(z')=\beta$.
Thus $\{\alpha,\beta\}\in G_r(y)\cdot e$. The orbit graph therefore
contains the complete bipartite graph between $\Desc_r(w)$ and
$\Desc_r(w')$ for every pair $w\neq w'$, and hence is connected.

\emph{Step 3: $\Cfrak_r$ implies $\Tfrak_r$.}
Let $\tau\in G_r(y)$ be a top-meridian element and let $e$ be one of
its transposition factors. For every $g\in G_r(y)$, the element
$g\tau g^{-1}$ belongs to the normal closure of $\tau$ and has $g(e)$
as a transposition factor. Hence the endpoints of every edge in
$G_r(y)\cdot e$ lie in the same orbit of that normal closure. By
$\Cfrak_r$, the graph is connected, so the normal closure acts
transitively. Starting from Step 1 and alternating Steps 2 and 3
proves $\Tfrak_r$ for all $r\geq1$. The preceding argument then
proves the lemma.
\end{proof}

We next determine the edge sets of the orbit graphs
in Definition~\ref{def:orbit-graph}.
Fix $m\geq1$ and $y\in U_m$.
Let $e=\{x,x'\}\subset T_m(y)$, with $x\ne x'$, and put
$j:=\min\{1\leq r\leq m\mid f^r(x)=f^r(x')\}$.

\begin{lem}
\label{lem:sc-pair-orbits}
The edge set of $\Gamma_m(e;y)$ is $\Omega_j$, where
\begin{equation}\label{eq:sc-pair-orbits}
 \Omega_j:=\bigl\{\{x,x'\}\subset T_m(y)\mid
 f^j(x)=f^j(x'),\quad
 f^{j-1}(x)\ne f^{j-1}(x')\bigr\}.
\end{equation}
\end{lem}

\begin{proof}
The sets $\Omega_1,\ldots,\Omega_m$ partition the set of unordered
pairs of distinct points in $T_m(y)$ and are $G_m(y)$-invariant by
Equation~\eqref{eq:tree-equivariance}.
Since $e\in\Omega_j$, its orbit is contained in $\Omega_j$.
By Lemma~\ref{lem:stabilizer-action},
$G_m(y)$ is transitive on $T_{m-j}(y)$, and the stabilizer of
$v\in T_{m-j}(y)$ induces the full group $G_j(v)$ on $T_j(v)$;
here $v\in U_j$. It therefore suffices to show that $(G_m(y))_v$
acts transitively on the pairs in $\Omega_j$ with common ancestor $v$.
Take two such pairs $\{a,b\}$ and $\{a',b'\}$.
In each pair, the two leaves have distinct ancestors in $T_1(v)$.
Since $G_j(v)\twoheadrightarrow G_1(v)=S_{d^k}$, there is
$\sigma\in G_j(v)$ sending the ancestors of $a,b$ in $T_1(v)$
to those of $a',b'$, respectively.
For $j\geq2$, Lemma~\ref{lem:two-coordinate-extraction} gives an element
$\kappa\in\ker(G_j(v)\to G_1(v))$ such that
$\kappa(\sigma(a))=a'$ and $\kappa(\sigma(b))=b'$.
Thus $\kappa\sigma$ sends $\{a,b\}$ to $\{a',b'\}$.
For $j=1$, these ancestors are the leaves themselves,
so $\sigma$ already sends $\{a,b\}$ to $\{a',b'\}$.
Thus $\Omega_j$ is one orbit.
\end{proof}

\begin{proof}[Proof of Proposition~\ref{prop:sc-monodromy-partitions}]
Use the notation of Lemma~\ref{lem:sc-pair-orbits}.
Let $\sim_R$ be such a relation. If it is the equality relation,
the assertion holds with $j=0$. Otherwise, let $j$ be the largest
index such that $a\sim_R b$ for some $\{a,b\}\in\Omega_j$.
Since $\sim_R$ is symmetric and $G_m(y)$-invariant and
$\Omega_j$ is one orbit, this holds for every $\{a,b\}\in\Omega_j$.
Fix $v\in T_{m-j}(y)$ and take $a,b\in T_j(v)$. If
$f^{j-1}(a)\ne f^{j-1}(b)$, then $\{a,b\}\in\Omega_j$, so $a\sim_R b$.
If $f^{j-1}(a)=f^{j-1}(b)$, choose
$c\in T_j(v)$ with $f^{j-1}(c)\ne f^{j-1}(a)=f^{j-1}(b)$.
Then
$\{a,c\},\{c,b\}\in\Omega_j$, and transitivity of $\sim_R$ gives $a\sim_R b$.
Therefore every fiber of $f^j:T_m(y)\to T_{m-j}(y)$
is contained in one $\sim_R$-class. Conversely, if $f^j(a)\ne f^j(b)$,
then $\{a,b\}\in\Omega_r$ for some $r>j$, so $a\sim_R b$ would
contradict the maximality of $j$.
We have proved that $a\sim_R b$ if and only if $f^j(a)=f^j(b)$
for a unique $0\leq j\leq m$.

\end{proof}

\section{Rational semiconjugacy rigidity and proof of Theorem~\ref{thm:main-semiconjugacy}}
\label{sec:semiconjugacy-rigidity}

\subsection{Decomposition and deck transformations}

In this subsection, we deduce from Section~\ref{sec:finite-level-monodromy}
two rigidity results concerning decompositions of iterates and deck transformations.

For a finite surjective morphism $a:X\to Z$ between normal
varieties, we write $\Rcal_a$ for the effective ramification
Weil divisor. We put
$B_a:=a(\Supp\Rcal_a)$, with reduced structure.
These definitions agree with the earlier notation when the
varieties are smooth.
If $Z$ is smooth, Zariski--Nagata purity
\cite[Exp.~X, Thm.~3.1]{SGA1} implies that $a$ is \'etale
over $Z\setminus B_a$.

\begin{thm}
\label{thm:sc-normal-factors}
Let $f\in\End_d^k$ be simple and let $m\geq1$. Suppose
\[
 f^m=P Q,
 \qquad Q:\Pbb^k\dashrightarrow Z,
 \qquad P:Z\dashrightarrow\Pbb^k,
\]
where $Z$ is normal, integral and projective and both maps are dominant rational.
Then there are an integer $0\leq j\leq m$ and a birational map
$\alpha:Z\dashrightarrow\Pbb^k$ such that
\[
 Q=\alpha^{-1}f^j,
 \qquad P=f^{m-j}\alpha.
\]
Moreover, if both $P$ and $Q$ are finite morphisms, then $\alpha$
is an isomorphism.
\end{thm}

\begin{proof}
Choose open subsets $Z^\circ\subset Z$ and
$V^\circ\subset\Pbb^k$ such that
$P:Z^\circ\to V^\circ$ is finite \'etale.
Let $X^\circ\subset\Pbb^k$ be an open subset on which $Q$
is defined and $Q(X^\circ)\subset Z^\circ$. Put
\[
 V:=(U_m\cap V^\circ)\setminus f^m(\Pbb^k\setminus X^\circ),
 \qquad Z_V:=Z^\circ\cap P^{-1}(V).
\]
We obtain a factorization by connected finite \'etale covers
\[
 f^{-m}(V)\xrightarrow{\,Q\,}Z_V
 \xrightarrow{\,P\,}V.
\]
For $y\in V$, the relation $x\sim_Q x'\Longleftrightarrow Q(x)=Q(x')$
on $T_m(y)$ is monodromy-invariant.
By Proposition~\ref{prop:sc-monodromy-partitions} and path lifting,
there is a unique integer $0\leq j\leq m$, independent of $y$, such that
\begin{equation}\label{eq:sc-Q-fibers}
 Q(x)=Q(x')
 \quad\Longleftrightarrow\quad
 f^j(x)=f^j(x')
 \qquad (x,x'\in T_m(y)).
\end{equation}

Let $\Gamma$ be the reduced closure of the image of $(Q,f^j)$ in
$Z\times\Pbb^k$.
The two projections $p_1:\Gamma\to Z$ and
$p_2:\Gamma\to\Pbb^k$ are birational by the equivalence above.
Set $\alpha:=p_2p_1^{-1}:Z\dashrightarrow\Pbb^k$.
By construction, $\alpha Q=f^j$ and $f^{m-j}\alpha=P$,
so the following diagram commutes, where dashed arrows denote rational maps:
\[
\begin{tikzcd}[column sep=large, row sep=large]
 & & Z \arrow[dd,dashed,"\alpha"] \arrow[dr,dashed,"P"] & \\
 \Pbb^k \arrow[r,dashed,"{(Q,f^j)}"]
 \arrow[urr,dashed,bend left=15,"Q"]
 \arrow[drr,bend right=15,"f^j"']
 & \Gamma \arrow[ur,"p_1"] \arrow[dr,"p_2"']
 & & \Pbb^k \\
 & & \Pbb^k \arrow[ur,"f^{m-j}"'] &
\end{tikzcd}
\]

If both $P$ and $Q$ are finite morphisms, the two projections $p_1,p_2$
are isomorphisms by Zariski's Main Theorem
\cite[Corollary 4.4.6]{Liu2002}. Hence $\alpha$ is an isomorphism.
\end{proof}

\begin{thm}
\label{thm:sc-deck-rigidity}
Let $f\in\End_d^k$ be simple and let $m\geq1$. If
$\delta\in\Aut(\Pbb^k)$ satisfies
\[
 f^m\delta=f^m,
\]
then $\delta=\id$.
\end{thm}

\begin{proof}
Fix a general $y\in U_m$. If $\delta$ is not the identity, choose $x\in T_m(y)$ with
$\delta x\ne x$ and put $e:=\{x,\delta x\}$.
Since the deck transformation $\delta$ commutes with $G_m(y)$,
transitivity implies that $\Gamma_m(e;y)$ has edge set
$\bigl\{\{a,\delta a\}\mid a\in T_m(y)\bigr\}$.
Every vertex $a$ therefore has at most two neighbors, namely
$\delta a$ and $\delta^{-1}a$.
On the other hand, Lemma~\ref{lem:sc-pair-orbits} shows that this
edge set is $\Omega_j$ for some $1\leq j\leq m$.
Hence every vertex has degree $N^{j-1}(N-1)\geq3$, where
$N:=d^k\geq4$.
This contradiction proves the claim.
\end{proof}

\subsection{Proof of Theorem~\ref{thm:main-semiconjugacy} when \texorpdfstring{$h$}{h} is finite}
\label{subsec:sc-regular}

Suppose that $h$ is finite. Since $Y$ is normal and $hu=f^mh$,
the rational map $u$ extends to a finite surjective morphism.

Put $F:=f^m$. Consider the fiber product
\[
 W:=\Pbb^k\times_{\Pbb^k}Y
   =\{(x,y)\in\Pbb^k\times Y\mid F(x)=h(y)\}.
\]
The projection $W\to Y$ is finite flat of degree $d^{km}$,
with reduced generic fiber. Hence $W$ is reduced and every
irreducible component dominates $Y$.
Let $\Gamma$ be the reduced image of $(h,u):Y\to W$,
and let $Y_1$ be its normalization.
The normalization map is finite, so $Y_1$ is a normal integral
projective variety and the projections induce finite surjective
morphisms $h_1:Y_1\to\Pbb^k$ and $v:Y_1\to Y$.
Since $Y$ is normal, $(h,u)$ lifts to a surjective finite morphism
$p:Y\to Y_1$ satisfying
\[
 h=h_1p,
 \qquad u=vp.
\]
Thus the following diagram commutes:
\[
\begin{tikzcd}[column sep=18mm, row sep=11mm]
 Y \arrow[r,"p"]
   \arrow[rrrr,bend left=25,"h"]
   \arrow[drrr,bend right=24,"u"']
 & Y_1 \arrow[r,"\text{normalization}"']
   \arrow[rrr,bend left=15,"h_1"]
   \arrow[drr,bend right=12,"v"']
 &[8mm] \Gamma \arrow[r,hook]
 & W \arrow[r,"\operatorname{pr}_1"']
   \arrow[d,"\operatorname{pr}_2"]
 & \Pbb^k \arrow[d,"F"] \\
 & & & Y \arrow[r,"h"'] & \Pbb^k
\end{tikzcd}
\]

In dimension one, with $Y=\Pbb^1$, the birationality condition in
the following lemma is equivalent to Pakovich's \emph{primitive}
condition; see~\cite[Section~6]{PakovichSemiconjugate2016}.

\begin{lem}
\label{lem:sc-ramification}
With the notation above,
if $(h,u):Y\to\Pbb^k\times Y$ is birational onto its image,
then $h$ is an isomorphism.
\end{lem}

\begin{proof}
By hypothesis, $p$ is birational and finite, hence an isomorphism.
The projection $\Gamma\to Y$ has topological degree
$\deg_{\mathrm{top}}v=\deg_{\mathrm{top}}u
=\deg_{\mathrm{top}}F=d^{km}$, by $Fh=hu$.
Thus $\Gamma$ is the only irreducible component of $W$,
and $(h,u)$ identifies $Y$ with the normalization of $W$.

We now consider the Cartesian square on the right of the diagram above.
Set $V:=\Pbb^k\setminus F^{-1}(B_h)$.
The base change $\operatorname{pr}_1:\operatorname{pr}_1^{-1}(V)\to V$
of $h$ is \'etale, so $\operatorname{pr}_1^{-1}(V)$ is normal.
Thus, writing $\nu:=(h,u)$ for the normalization,
\[
 h:\;h^{-1}(V)
 \xrightarrow[\sim]{\nu}
 \operatorname{pr}_1^{-1}(V)
 \xrightarrow[\text{\'etale}]{\operatorname{pr}_1}V
\]
is \'etale.
Thus $B_h\subseteq F^{-1}(B_h)$ and $F(B_h)\subseteq B_h$.
In particular, by Lemma~\ref{lem:ram-level-separation},
$B_h$ and $D_m$ have no common irreducible component.

Since $F$ is \'etale over $U_m$, its base change by $h$ is \'etale
over the normal variety $h^{-1}(U_m)$.
Hence $W$ is normal over $F^{-1}(U_m)$.
Thus normalization is an isomorphism there,
and \'etale base change gives
\[
 B_h\cap F^{-1}(U_m)=F^{-1}(B_h\cap U_m).
\]
Taking closures gives
$F^{-1}(B_h)\subseteq B_h$. Together with the opposite inclusion,
this yields $F^{-1}(B_h)=B_h$.
Corollary~\ref{cor:no-totally-invariant-hypersurface} now gives
$B_h=\varnothing$. Hence $h$ is finite \'etale and therefore
an isomorphism, since $\Pbb^k$ is simply connected.
\end{proof}

\begin{proof}[Proof of Theorem~\ref{thm:main-semiconjugacy} when $h$ is finite]
We prove existence by induction on $\deg_{\mathrm{top}}h$,
uniformly over all $m$ and all relations $f^m h=hu$.
If $\deg_{\mathrm{top}}h=1$, then $h$ is an isomorphism and the
result follows with $s=0$ and $\mu=h$.

Assume $\deg_{\mathrm{top}}h>1$. Then $\deg_{\mathrm{top}}p>1$
by Lemma~\ref{lem:sc-ramification}. Set $u_1:=pv$.
Since $Fh_1p=h_1pvp$ and $h=h_1p$, we have
\[
 Fh_1=h_1u_1,
 \qquad \deg_{\mathrm{top}}h_1<\deg_{\mathrm{top}}h.
\]
By induction, there are an integer $s_1\geq0$ and an
isomorphism $\mu_1:Y_1\to\Pbb^k$ with
\[
 h_1=f^{s_1}\mu_1,
 \qquad u_1=\mu_1^{-1}F\mu_1.
\]
Consequently, $F=\mu_1u_1\mu_1^{-1}=(\mu_1p)(v\mu_1^{-1})$.
Theorem~\ref{thm:sc-normal-factors} gives an integer $0\leq j\leq m$ and an isomorphism
$\mu:Y\to\Pbb^k$ such that
\[
 v\mu_1^{-1}=\mu^{-1}f^j,
 \qquad \mu_1p=f^{m-j}\mu.
\]
Therefore
\[
 h=h_1p=f^{s_1+m-j}\mu,
 \qquad
 u=vp=\mu^{-1}f^j f^{m-j}\mu=\mu^{-1}f^m\mu.
\]
Setting $s:=s_1+m-j\geq0$ proves existence.

Finally, $\deg_{\mathrm{top}}h=d^{ks}$ determines $s$ uniquely, and
the identity $h=f^s\mu$ determines $\mu$ uniquely: this is immediate
when $s=0$ and follows from Theorem~\ref{thm:sc-deck-rigidity}
when $s\geq1$.
\end{proof}

\subsection{Proof of Theorem~\ref{thm:main-semiconjugacy} in general}
\label{sec:dominant-rational-maps}

\begin{proof}[Proof of Theorem~\ref{thm:main-semiconjugacy}]
Let $I_h$ denote the indeterminacy locus of $h$, and put
\[
 \Gamma_h:=\overline{\{(x,h(x))\mid x\in Y\setminus I_h\}}
 \subset Y\times\Pbb^k.
\]
Let $X$ be the normalization of $\Gamma_h$. The two projections
induce morphisms $\pi:X\to Y$ and $q:X\to\Pbb^k$, with
$\pi$ birational and $q=h\pi$. The Stein factorization of $q$ gives
$X\overset{r}{\to}Z\overset{p}{\to}\Pbb^k$,
where $Z$ is normal and integral, $r$ is birational, and $p$ is finite
surjective. Thus $h=p\nu$ for the birational map
$\nu:=r\pi^{-1}:Y\dashrightarrow Z$.
Define the dominant rational map
$\widetilde u:=\nu u\nu^{-1}:Z\dashrightarrow Z$.
By $f^m h=hu$ and $h=p\nu$, we have
\[
 f^m p=p\widetilde u,
 \qquad
 \widetilde u\nu=\nu u.
\]
Thus $\widetilde u$ is a finite surjective morphism.
The construction gives the following commutative diagram,
where dashed arrows denote rational maps:
\[
\begin{tikzcd}[column sep=22mm, row sep=18mm]
 Y \arrow[d,dashed,"u"']
   \arrow[rr,dashed,bend right=22,"\nu"']
 & X \arrow[l,"\pi"'] \arrow[r,"r"]
   \arrow[rr,bend left=26,"q"]
 & Z \arrow[r,"p"] \arrow[d,"\widetilde u"']
 & \Pbb^k \arrow[d,"f^m"] \\
 Y \arrow[rr,dashed,"\nu"']
   \arrow[rrr,dashed,bend right=26,"h"']
 &
 & Z \arrow[r,"p"']
 & \Pbb^k
\end{tikzcd}
\]
Apply the finite case proved in Section~\ref{subsec:sc-regular}
to $p$ and $\widetilde u$. There are an integer $s\geq0$ and an isomorphism
$\alpha:Z\to\Pbb^k$ such that
\[
 p=f^s\alpha,
 \qquad
 \widetilde u=\alpha^{-1}f^m\alpha.
\]
Set $\mu:=\alpha\nu$. Then $\mu$ is birational and
\begin{equation}\label{eq:rational-birational-normal-form}
 h=f^s\mu,
 \qquad
 \mu u=f^m\mu.
\end{equation}
Any birational map $\delta:\Pbb^k\dashrightarrow\Pbb^k$
satisfying $f^s\delta=f^s$ extends to an automorphism.\linebreak
Uniqueness then follows as in the finite case (Section~\ref{subsec:sc-regular}).

It remains to prove the additional assertions.
If $h$ is finite, then $X=Y$ and $r$ is finite birational,
hence an isomorphism. Thus $\nu$ and $\mu$ are isomorphisms.

Suppose now that $Y=\Pbb^k$ and that $u$ is a morphism.
Put $\psi:=\mu^{-1}$ and let $E:=\operatorname{Exc}(\psi)$
be the union of the hypersurfaces contracted by $\psi$.
Since $\psi f^m=u\psi$, we have $(f^m)^{-1}(E)=E$.
Corollary~\ref{cor:no-totally-invariant-hypersurface} now gives
$E=\varnothing$, so $\psi,\mu\in\PGL_{k+1}(\Cbb)$.
\end{proof}

\section{Iterated centralizers and proof of Theorems~\ref{thm:main} and~\ref{thm:regular-polynomial-centralizer}}
\label{sect:proof}

\subsection{Proof of Theorem~\ref{thm:main}}

The rational semiconjugacy theorem reduces the classification of
commuting rational maps to automorphisms of iterates.
For an endomorphism $F$ of $\Pbb^k$, define its
\emph{commuting automorphism group} by
$\Aut(F):=\{\alpha\in\Aut(\Pbb^k)\mid\alpha F=F\alpha\}$.

Recall that
\[
 W_d=H^0(\Pbb^k,\Ocal_{\Pbb^k}(d))
 \simeq\Sym^d(\Cbb^{k+1})^\vee.
\]
Write an endomorphism $f\in\End_d^k$ as
$f=[F_0:\cdots:F_k]$, and put
\[
 F(z):=(F_0(z),\ldots,F_k(z))^{\mathsf t},
 \qquad
 L_f:=\langle F_0,\ldots,F_k\rangle\subset W_d.
\]
Since $f$ is finite and surjective, the forms
$F_0,\ldots,F_k$ are linearly independent. For
$\alpha=[A]\in\PGL_{k+1}$ and $L\subset W_d$, set
$\alpha^*L:=\{P(Az)\mid P\in L\}$.
This subspace is independent of the chosen lift $A$. Define
\[
 H_f:=\{\alpha\in\PGL_{k+1}\mid\alpha^*L_f=L_f\}.
\]
If $\alpha=[A]\in H_f$, choose a lift $A\in\GL_{k+1}$. There is a
unique $\widehat A_f\in\GL_{k+1}$ such that
$F(Az)=\widehat A_fF(z)$.
Replacing $A$ by a scalar multiple changes $\widehat A_f$ by a scalar
multiple. Hence
\[
 \kappa_f:H_f\longrightarrow\PGL_{k+1},
 \qquad
 \kappa_f(\alpha):=[\widehat A_f]
\]
is well defined. The preceding identity also shows that $\kappa_f$ is
a homomorphism and that $f\alpha=\kappa_f(\alpha)f$.

We need the following proposition. The case $(d,k)\ne(2,2)$ is due to
Guralnick--Lawther~\cite[Theorem~4 and Table~1.4]{GuralnickLawther2024}.

\begin{prop}\label{prop:generic-rigidity}
\textup{\cite[Proposition~7.5]{ZhangMeasureRigidity}}
There exists a dense Zariski open subset $V^{\mathrm{aut}}_{d,k}
 \subset \End_d^k$ defined over $\Qbb$ such that, for every $f\in V^{\mathrm{aut}}_{d,k}(\Cbb)$,
\[
        \kappa_f^{-1}(H_f)=\{\id\}.
\]
\end{prop}

\begin{defi}\label{defi:strongly-simple}
A simple endomorphism $f\in\End_d^k$ satisfying
$\kappa_f^{-1}(H_f)=\{\id\}$ is called \emph{strongly simple}.
\end{defi}

Set $U_{d,k}:=
U_{\mathrm{simp}}\cap V^{\mathrm{aut}}_{d,k}$.

\begin{thm}\label{thm:aut-iterate-trivial}
The subset $U_{d,k}\subset
\End_d^k$ is dense Zariski open and defined over $\Qbb$. For every
$f\in U_{d,k}(\Cbb)$ and every $m\geq1$,
\[
        \Aut(f^m)=\{\id\}.
\]
\end{thm}

\begin{proof}
We first show that $Af^\ell=f^\ell B$, with
$A,B\in\Aut(\Pbb^k)$ and $\ell\geq1$, implies $B\in H_f$.
Indeed, $f^\ell=(A^{-1}f^{\ell-1})(fB)$.
Theorem~\ref{thm:sc-normal-factors} gives $fB=\varepsilon f^j$
for some automorphism $\varepsilon$ and some $0\leq j\leq\ell$.
Comparing degrees gives $j=1$, so $fB=\varepsilon f$,
and hence $B\in H_f$.

Now let $\alpha\in\Aut(f^m)$. Applying the preceding observation with
$A=B=\alpha$ and $\ell=m$ gives $\alpha\in H_f$. Put
$\gamma:=\kappa_f(\alpha)$, so that $f\alpha=\gamma f$.
If $m=1$, then $\alpha f=f\alpha=\gamma f$; since $f$ is
surjective, $\gamma=\alpha\in H_f$. If $m\geq2$, then
$\alpha f^m=f^m\alpha=f^{m-1}(f\alpha)=f^{m-1}\gamma f$.
Cancelling the final surjective factor $f$ gives
$\alpha f^{m-1}=f^{m-1}\gamma$. The same observation gives
$\gamma\in H_f$. Thus in either case $\alpha\in H_f$ and
$\kappa_f(\alpha)=\gamma\in H_f$. The rigidity condition gives
$\alpha=\id$.
\end{proof}

\begin{proof}[Proof of Theorem~\ref{thm:main}]
Fix $f\in U_{d,k}(\Cbb)$. The inclusion $\langle f\rangle\subseteq C_{\mathrm{rat}}(f^\infty)$
is immediate. Conversely, suppose that $h f^m=f^m h$ for some
$m\geq1$.
Apply Theorem~\ref{thm:main-semiconjugacy} with
$u=f^m$. Then
\[
 h=f^s\alpha,
 \qquad
 \alpha f^m=f^m\alpha
\]
for some $s\geq0$ and $\alpha\in\PGL_{k+1}(\Cbb)$. By
Theorem~\ref{thm:aut-iterate-trivial},
$\alpha\in\Aut(f^m)=\{\id\}$.
Therefore $h=f^s\in\langle f\rangle$.
\end{proof}

\subsection{Proof of Theorem~\ref{thm:regular-polynomial-centralizer}}
\label{sec:regular-polynomial}

The affine centralizer is controlled by the \emph{image--fiber
principle} from the author's earlier work on the rigidity of equilibrium
measures of regular polynomial endomorphisms~\cite{ZhangMeasureRigidity}.
For $e\geq2$ and
$f\in\Poly_e^k$, its
\emph{Green function} is
\[
 G_f(z):=\lim_{n\to\infty}e^{-n}\log^+\|f^n(z)\|.
\]
It is the unique continuous nonnegative plurisubharmonic function
$G$ on $\Cbb^k$ satisfying the two properties
\[
 G(z)-\log^+\|z\|=O(1)\quad\text{as }\|z\|\to\infty,
 \qquad G\circ f=eG.
\]
The \emph{equilibrium measure} and its Julia set are
\[
 \mu_f:=(dd^cG_f)^k,
 \qquad J_f:=\supp(\mu_f).
\]

Let
$f\in\Poly_d^k$, and put $\Pi:=\Pbb^k\setminus\Cbb^k\simeq\Pbb^{k-1}$.
Then $f^{-1}(\Pi)=\Pi$, and the restriction of $f$ to $\Pi$ is an
endomorphism $f_\Pi\in\End_d(\Pi)$. Set
\[
 C_{\mathrm{reg}}(f^N)
 :=\{g\in\bigsqcup_{e\geq1}\Poly_e^k
      \mid gf^N=f^Ng\},
 \qquad
 C_{\mathrm{reg}}(f^\infty)
 :=\bigcup_{N\geq1}C_{\mathrm{reg}}(f^N).
\]

\begin{thm}[Measure rigidity
 {\cite[Theorems~1.2 and~1.4]{ZhangMeasureRigidity}}]
\label{thm:measure-rigidity-input}
Let $f,g\in\Poly_d^k$ satisfy $f_\Pi=g_\Pi$.
Then
\[
 \mu_f=\mu_g\iff g=\sigma f
\]
for an affine automorphism $\sigma$ of $\Cbb^k$ that preserves
$J_f$ and induces the identity on $\Pi$.

Moreover, there exists a dense Zariski open subset
$U^{\mathrm{MR}}_{d,k}\subset\Poly_d^k$, defined over $\Qbb$, such
that, for every $f\in U^{\mathrm{MR}}_{d,k}(\Cbb)$,
\[
 \{g\in\Poly_d^k\mid \mu_g=\mu_f\}=\{f\}.
\]
If $k=2$, then moreover
\[
 \left\{g\in\bigsqcup_{e\geq2}\Poly_e^2\mid \mu_g=\mu_f\right\}
 =\{f^n\mid n\geq1\}.
\]
\end{thm}

\begin{prop}\label{prop:regular-affine-centralizer}
There exists a dense Zariski open subset $U^{\mathrm{aff}}_{d,k}
        \subset\Poly_d^k$ defined over $\Qbb$
such that for every $f\in U^{\mathrm{aff}}_{d,k}(\Cbb)$,
\[
        C_{\mathrm{reg}}(f^\infty)=\langle f\rangle.
\]
\end{prop}

\begin{proof}
We begin with an identity valid without a genericity assumption.
Let $f\in\Poly_d^k$, and let $g$ be a degree-$e$ regular polynomial
endomorphism commuting with $f^N$.
Write $G_f$ for the Green function of $f$. The function
$e^{-1}G_f\circ g$ has logarithmic growth and satisfies
\[
        (e^{-1}G_f\circ g)\circ f^N
        =d^N(e^{-1}G_f\circ g).
\]
By the uniqueness of the Green function, applied to $f^N$ whose
Green function is $G_f$,
\begin{equation}\label{eq:regular-poly-green}
        G_f\circ g=eG_f.
\end{equation}
If $e\geq2$, applying the same uniqueness statement to $g$ gives
$G_g=G_f$, and hence $\mu_g=\mu_f$.

There are two cases, since Theorem~\ref{thm:main} applies on
$\Pi\simeq\Pbb^{k-1}$ only when $k\geq3$.
Suppose first that $k=2$. Set
$U^{\mathrm{aff}}_{d,2}:=U^{\mathrm{MR}}_{d,2}$ and assume
$f\in U^{\mathrm{aff}}_{d,2}(\Cbb)$. If $e\geq2$, the last assertion
of Theorem~\ref{thm:measure-rigidity-input} gives $g=f^m$ for some
$m\geq1$. If $e=1$,
put $h:=gf$. Equation~\eqref{eq:regular-poly-green} gives
$G_f\circ h=dG_f$, so $G_h=G_f$ and $\mu_h=\mu_f$. The generic
conclusion for maps of the same degree gives $h=f$; the surjectivity of $f$
then gives
$g=\id$.

Assume now that $k\geq3$. Let
$U^{\mathrm{proj}}_{d,k-1}\subset\End_d(\Pi)$ be the dense
Zariski open subset, defined over $\Qbb$, supplied by
Theorem~\ref{thm:main} in dimension $k-1$. The restriction morphism
\[
 \rho_{d,k}:\Poly_d^k
 \longrightarrow\End_d(\Pi),
 \qquad f\longmapsto f_\Pi,
\]
is surjective and defined over $\Qbb$. Set
\[
 U^{\mathrm{aff}}_{d,k}
 :=U^{\mathrm{MR}}_{d,k}
   \cap\rho_{d,k}^{-1}(U^{\mathrm{proj}}_{d,k-1}).
\]
This is a dense Zariski open subset of $\Poly_d^k$ defined over
$\Qbb$: indeed, $\rho_{d,k}$ is surjective and $\Poly_d^k$ is
irreducible. Assume $f\in U^{\mathrm{aff}}_{d,k}(\Cbb)$. Restriction to
$\Pi$ gives $g_\Pi f_\Pi^N=f_\Pi^N g_\Pi$.
Theorem~\ref{thm:main}, applied on $\Pi\simeq\Pbb^{k-1}$, yields $g_\Pi=f_\Pi^m$
for some $m\geq0$, and therefore $e=d^m$.

If $m\geq1$, then $G_g=G_f=G_{f^m}$, and the leading homogeneous
parts of $g$ and $f^m$ differ by a nonzero scalar. By the first
assertion of Theorem~\ref{thm:measure-rigidity-input}, $g=\sigma f^m$
for an affine automorphism $\sigma$ preserving
$J_{f^m}=J_f$. Put $h:=\sigma f$. The reverse implication in
the same assertion gives $\mu_h=\mu_f$. Since $h$ and $f$ both have
degree $d$, the generic conclusion of
Theorem~\ref{thm:measure-rigidity-input} gives $h=f$.
Surjectivity of $f$ gives $\sigma=\id$, and hence $g=f^m$.

If $m=0$, then $e=1$. Put again $h:=gf$. As in the
case $k=2$, \eqref{eq:regular-poly-green} gives
$G_h=G_f$, and the generic conclusion for maps of the same degree in
Theorem~\ref{thm:measure-rigidity-input} gives $h=f$.
Thus $g=\id$.
The reverse inclusion $\langle f\rangle\subset
C_{\mathrm{reg}}(f^\infty)$ is immediate.
\end{proof}

We use the following consequence of Dinh--Sibony's genericity
result~\cite[Lemma~1.69 and Theorem~1.47]{DSbook}.

\begin{lem}\label{lem:unique-infinity}
There exists a dense Zariski open subset $V^{\mathrm{inv}}_{d,k}
\subset\Poly_d^k$, defined over $\Qbb$, such that, for every
$f\in V^{\mathrm{inv}}_{d,k}(\Cbb)$ and every $N\geq1$, the map
$f_\Pi^N$ has no nonempty proper totally invariant algebraic subset
of $\Pi$. Moreover, if a nonempty hypersurface $D\subset\Pbb^k$
satisfies $(f^N)^{-1}(D)=D$, then $D=\Pi$.
\end{lem}

\begin{proof}[Proof of Theorem~\ref{thm:regular-polynomial-centralizer}]
Set $U^{\mathrm{poly}}_{d,k}
 :=V^{\mathrm{inv}}_{d,k}\cap U^{\mathrm{aff}}_{d,k}.$
Let $f\in U^{\mathrm{poly}}_{d,k}(\Cbb)$, and let
$g:\Pbb^k\dashrightarrow\Pbb^k$ be a dominant rational map
commuting with $f^N$ for some $N\geq1$.

Put $E:=g^*\Pi$, and let $I_g$ be the indeterminacy locus of $g$.
Since $f^N$ is finite, commutation gives
\[
 (f^N)^*E=d^NE,\qquad (f^N)^{-1}(I_g)=I_g.
\]
Lemma~\ref{lem:unique-infinity} gives $\Supp(E)=\Pi$, hence
$I_g\subset\Pi$. Since $(f_\Pi^N)^{-1}(I_g)=I_g$, the same lemma
gives $I_g=\varnothing$. Thus $g$ is a regular polynomial endomorphism,
and Proposition~\ref{prop:regular-affine-centralizer} gives $g=f^j$
for some $j\geq0$.
The reverse inclusion is immediate.
\end{proof}

\section{Periodic correspondences and proof of Theorem~\ref{thm:periodic-correspondences}}
\label{sec:periodic-correspondences}

\subsection{Periodic self-correspondences}

For a morphism $h:\Pbb^k\to\Pbb^k$, write
\[
 \Gamma_h:=\{(x,h(x))\mid x\in\Pbb^k\},
 \qquad
 \Gamma_h^{\mathrm{op}}:=\{(h(x),x)\mid x\in\Pbb^k\}.
\]

\begin{lem}\label{lem:periodic-self-correspondences}
Let $f\in U_{d,k}(\Cbb)$. If a dominant correspondence
$W\subset\Pbb^k\times\Pbb^k$ satisfies
$(f^m\times f^m)(W)=W$ for some $m\geq1$, then
\[
 W=\Gamma_{f^r}\quad\text{or}\quad W=\Gamma_{f^r}^{\mathrm{op}}
 \qquad\text{for some }r\geq0.
\]
\end{lem}

\begin{proof}
Let $\nu:Y\to W$ be the normalization and let $p,q:Y\to\Pbb^k$
be the two projections. The restriction of $f^m\times f^m$ to $W$
induces a unique finite surjective morphism $u:Y\to Y$ such that
\[
 pu=f^m p,\qquad qu=f^m q.
\]
Theorem~\ref{thm:main-semiconjugacy} gives an integer $a\geq0$ and a birational map $\mu:Y\dashrightarrow\Pbb^k$ such that
\[
 p=f^a\mu,\qquad \mu u=f^m\mu.
\]
The dominant rational map $H:=q\mu^{-1}$ therefore satisfies
\[
 Hf^m=q\mu^{-1}f^m=qu\mu^{-1}=f^m q\mu^{-1}=f^m H.
\]
By Theorem~\ref{thm:main}, $H=f^b$ for some $b\geq0$.
It follows that
\begin{equation}\label{eq:self-correspondence-parametrization}
 W=\bigl\{(f^a(z),f^b(z))\mid z\in\Pbb^k\bigr\}.
\end{equation}
Surjectivity of $f^{\min\{a,b\}}$ gives the conclusion with $r=|a-b|$.
\end{proof}

\subsection{Rational dynamical quotients}

\begin{thm}\label{thm:rational-dynamical-quotients}
Let $f\in U_{d,k}(\Cbb)$ and $m\geq1$.
Let $Y$ be a normal integral projective variety of dimension $k$.
Suppose that $H:\Pbb^k\dashrightarrow Y$ and
$v:Y\dashrightarrow Y$ are dominant rational maps satisfying
\begin{equation}\label{eq:rational-dynamical-quotient}
 Hf^m=vH.
\end{equation}
Then there are a unique integer $s\geq0$ and a unique
birational map $\alpha:\Pbb^k\dashrightarrow Y$ such that
\begin{equation}\label{eq:rational-quotient-normal-form}
 H=\alpha f^s,\qquad v=\alpha f^m\alpha^{-1}.
\end{equation}
Moreover, if $H$ is a finite morphism, then $\alpha$ is an isomorphism.
If $Y=\Pbb^k$ and $v$ is a morphism, then
$\alpha\in\PGL_{k+1}(\Cbb)$.
\end{thm}

\begin{proof}
Put $F:=f^m$.
Since $H$ is generically finite and we work in characteristic zero,
there are dense open subsets $U\subset\Pbb^k$ and $V\subset Y$ such that
$H|_U:U\to V$ is finite \'etale.
Let $\Ecal_H$ be the reduced closure of
$\{(x,y)\in U\times U\mid H(x)=H(y)\}$ in
$\Pbb^k\times\Pbb^k$. Each irreducible component of $\Ecal_H$
has dimension $k$ and dominates both factors.

For a general point $(x,y)$ of any such component,
\eqref{eq:rational-dynamical-quotient} gives
\[
 H(x)=H(y)
 \quad\Longrightarrow\quad
 H(F(x))=v(H(x))=v(H(y))=H(F(y)).
\]
As $F\times F$ is finite, it sends each irreducible component of
$\Ecal_H$ onto an irreducible component of $\Ecal_H$.
There are only finitely many components, so each eventually maps to
a periodic component. By Lemma~\ref{lem:periodic-self-correspondences},
every periodic component is $\Gamma_{f^r}$ or
$\Gamma_{f^r}^{\mathrm{op}}$ for some $r\geq0$.
Containment of either graph in $\Ecal_H$ implies
$Hf^r=H$ as rational maps, and hence $r=0$.
Thus the diagonal $\Delta$ is the only periodic component of
$\Ecal_H$. We can choose a single integer $n\geq1$ such that
\begin{equation}\label{eq:quotient-fibers-collapse}
 (F^n\times F^n)(\Ecal_H)\subseteq\Delta.
\end{equation}

Put $G:=f^{mn}$. Equation~\eqref{eq:quotient-fibers-collapse}
says that $G$ is constant on a general fiber of $H$.
Let
\[
 \Gamma:=\overline{\{(H(x),G(x))\mid x\in U\}}
 \subset Y\times\Pbb^k
\]
with reduced structure, and let $p:\Gamma\to Y$ and $q:\Gamma\to\Pbb^k$
be the two projections.
By \eqref{eq:quotient-fibers-collapse}, the first projection from
$\Gamma$ has a single point over a general point of $Y$; thus $p$ is birational.
Set $R:=qp^{-1}:Y\dashrightarrow\Pbb^k$. Then $f^{mn}=RH$.
By Theorem~\ref{thm:sc-normal-factors},
there are $0\leq s\leq mn$ and a birational map
$\alpha:\Pbb^k\dashrightarrow Y$ such that $H=\alpha f^s$.
Substituting in \eqref{eq:rational-dynamical-quotient} and cancelling
the dominant right factor $f^s$ yields
\[
 \alpha f^m=v\alpha.
\]
Thus \eqref{eq:rational-quotient-normal-form} holds.
The equality $\deg_{\mathrm{top}}(H)=d^{ks}$ determines $s$ uniquely,
and $H=\alpha f^s$ determines $\alpha$ uniquely since $f^s$ is surjective.

If $H$ is finite, then $\Gamma$ is the image of the morphism $(H,G)$,
so $p$ is finite. Since $Y$ is normal, $p$ is an isomorphism.
Thus $R$ is a morphism; it is finite because $RH=f^{mn}$ is finite
and $H$ is surjective. The finite case of
Theorem~\ref{thm:sc-normal-factors} shows that $\alpha$ is an isomorphism.

Finally, if $Y=\Pbb^k$ and $v$ is a morphism, the last assertion of
Theorem~\ref{thm:main-semiconjugacy}, applied to
$h=\alpha^{-1}$ and $u=v$, together with the birationality of
$\alpha^{-1}$ gives $\alpha\in\PGL_{k+1}(\Cbb)$, so $H$ is a morphism.
\end{proof}

\subsection{Mixed periodic correspondences}

\begin{proof}[Proof of Theorem~\ref{thm:periodic-correspondences}]
Put $\Phi:=f\times g$.
Suppose that $\Phi^m(Z)=Z$ for some $m\geq1$.
As in the proof of Lemma~\ref{lem:periodic-self-correspondences},
we obtain a dominant rational map $H:\Pbb^k\dashrightarrow\Pbb^k$
satisfying $Hf^m=g^mH$.
Then Theorem~\ref{thm:rational-dynamical-quotients} gives
$\alpha\in\PGL_{k+1}(\Cbb)$ such that $g^m=\alpha f^m\alpha^{-1}$.
By Theorem~\ref{thm:main}, $g=\alpha f\alpha^{-1}$.
Set $W:=(\id\times\alpha^{-1})(Z)$. Then $(f^m\times f^m)(W)=W$.
Lemma~\ref{lem:periodic-self-correspondences} gives
$W=\Gamma_{f^r}$ or $W=\Gamma_{f^r}^{\mathrm{op}}$ for some $r\geq0$.
Thus $Z=(\id\times\alpha)(W)$ is one of the graphs
in \eqref{eq:periodic-correspondence-graphs}.
Conversely, both graphs are invariant under $\Phi$.
\end{proof}

We obtain the following consequence for the diagonal.
Set $\Delta:=\{(x,x)\mid x\in\Pbb^k\}$.

\begin{cor}
\label{cor:preperiodic-diagonal}
Let $f\in U_{d,k}(\Cbb)$, and let
$g:\Pbb^k\to\Pbb^k$ be any endomorphism of degree at least two.
Then $\Delta$ is preperiodic under $f\times g$ if and only if $f=g$.
\end{cor}

\begin{proof}
If $f=g$, then $\Delta$ is invariant.
Conversely, if $\Delta$ is preperiodic, choose $n\geq0$ so that
$W:=(f\times g)^n(\Delta)$ is periodic.
Theorem~\ref{thm:periodic-correspondences} gives
$g=\alpha f\alpha^{-1}$ and either
\[
 g^n=\alpha f^{n+r}\qquad\text{or}\qquad
 f^n=f^r\alpha^{-1}g^n
\]
for some $r\geq0$ and $\alpha\in\PGL_{k+1}(\Cbb)$.
Comparing degrees gives $r=0$, hence $g^n=\alpha f^n$.
Thus $f^n=f^n\alpha^{-1}$, which gives $\alpha=\id$ by
Theorem~\ref{thm:sc-deck-rigidity} when $n\geq1$, and directly when $n=0$.
Thus $g=f$.
\end{proof}

\bibliographystyle{plain}
\bibliography{Mybio2}

@article{Smale1998,
  author  = {Smale, Steve},
  title   = {Mathematical problems for the next century},
  journal = {The Mathematical Intelligencer},
  volume  = {20},
  number  = {2},
  pages   = {7--15},
  year    = {1998},
  doi     = {10.1007/BF03025291}
}

@article{BonattiCrovisierWilkinson2009,
  author  = {Bonatti, Christian and Crovisier, Sylvain and Wilkinson, Amie},
  title   = {The {$C^1$} generic diffeomorphism has trivial centralizer},
  journal = {Publications Math{\'e}matiques de l'Institut des Hautes {\'E}tudes Scientifiques},
  volume  = {109},
  pages   = {185--244},
  year    = {2009},
  doi     = {10.1007/s10240-009-0021-z}
}

@misc{beaumont2025centralizersendomorphismsprojectiveline,
      title={On the centralizers of endomorphisms of the projective line}, 
      author={Beaumont, Alonso},
      year={2025},
      howpublished={Preprint, arXiv:\allowbreak 2512.13523},
      eprint={2512.13523},
      archivePrefix={arXiv},
      primaryClass={math.DS},
      url={https://arxiv.org/abs/2512.13523}, 
}

@misc{GT24,
 author = {Igors Gorbovickis and Johan Taflin},
 title = {Independence of multipliers in several variables complex dynamics},
 year = {2024},
 howpublished = {Preprint, {arXiv}:2411.12856 [math.{DS}]},
 url = {https://arxiv.org/abs/2411.12856},
 arXiv = {arXiv:2411.12856}
}

@article{Ritt1923,
  author  = {Ritt, J. F.},
  title   = {Permutable rational functions},
  journal = {Transactions of the American Mathematical Society},
  volume  = {25},
  number  = {3},
  pages   = {399--448},
  year    = {1923},
  doi     = {10.2307/1989018}
}

@article{Pak20,
 author = {Pakovich, Fedor},
 title = {Finiteness theorems for commuting and semiconjugate rational functions},
 fjournal = {Conformal Geometry and Dynamics},
 journal = {Conform. Geom. Dyn.},
 issn = {1088-4173},
 volume = {24},
 pages = {202--229},
 year = {2020},
 language = {English},
 doi = {10.1090/ecgd/354},
 zbMATH = {7271858},
 Zbl = {1451.30053}
}

@article{Pak21,
 author = {Pakovich, Fedor},
 title = {Commuting rational functions revisited},
 fjournal = {Ergodic Theory and Dynamical Systems},
 journal = {Ergodic Theory Dyn. Syst.},
 issn = {0143-3857},
 volume = {41},
 number = {1},
 pages = {295--320},
 year = {2021},
 language = {English},
 doi = {10.1017/etds.2019.51},
 zbMATH = {7282579},
 Zbl = {1461.30059}
}

@incollection{DSbook,
 author = {Dinh, Tien-Cuong and Sibony, Nessim},
 title = {Dynamics in several complex variables: endomorphisms of projective spaces and polynomial-like mappings},
 booktitle = {Holomorphic Dynamical Systems},
 editor = {Gentili, Graziano and Guenot, Jacques and Patrizio, Giorgio},
 series = {Lecture Notes in Mathematics},
 volume = {1998},
 isbn = {978-3-642-13170-7; 978-3-642-13171-4},
 pages = {165--294},
 year = {2010},
 publisher = {Springer},
 address = {Berlin},
 doi = {10.1007/978-3-642-13171-4_4},
 language = {English},
 zbMATH = {5879453},
 Zbl = {1218.37055}
}

@incollection{sibonybook,
 author = {Sibony, Nessim},
 title = {Dynamique des applications rationnelles de {$\mathbb{P}^k$}},
 booktitle = {Dynamique et g\'eom\'etrie complexes},
 series = {Panoramas et Synth\`eses},
 volume = {8},
 isbn = {2-85629-078-7},
 pages = {97--185},
 year = {1999},
 publisher = {Soci{\'e}t{\'e} Math{\'e}matique de France},
 address = {Paris},
 language = {French},
 zbMATH = {1908328},
 Zbl = {1020.37026}
}

@article{DinhCommuting,
 author = {Dinh, Tien-Cuong},
 title = {Sur les endomorphismes polynomiaux permutables de {${\mathbb C}^2$}},
 fjournal = {Annales de l'Institut Fourier},
 journal = {Ann. Inst. Fourier},
 issn = {0373-0956},
 volume = {51},
 number = {2},
 pages = {431--459},
 year = {2001},
 language = {French},
 doi = {10.5802/aif.1828},
 url = {https://eudml.org/doc/115921},
 zbMATH = {1584199},
 Zbl = {0977.30016}
}

@article{AguilarAguilar,
 author = {Aguilar Aguilar, Rodolfo},
 title = {The fundamental group of partial compactifications of the complement of a real line arrangement},
 fjournal = {Topology and its Applications},
 journal = {Topology Appl.},
 issn = {0166-8641},
 volume = {283},
 pages = {107388},
 year = {2020},
 language = {English},
 doi = {10.1016/j.topol.2020.107388},
 zbMATH = {7285219},
 Zbl = {1460.57023}
}

@misc{pakovich2026periodiccurvesgeneralendomorphisms,
      title={Periodic curves for general endomorphisms of {${\mathbb C}{\mathbb P}^1\times {\mathbb C}{\mathbb P}^1$}}, 
      author={Pakovich, Fedor},
      year={2025},
      howpublished={Preprint, arXiv:\allowbreak 2506.09948},
      eprint={2506.09948},
      archivePrefix={arXiv},
      primaryClass={math.DS},
      url={https://arxiv.org/abs/2506.09948}, 
}

@article{pakovich-deg4,
 author = {Pakovich, Fedor},
 title = {On iterates of rational functions with maximal number of critical values},
 fjournal = {Journal d'Analyse Math{\'e}matique},
 journal = {J. Anal. Math.},
 issn = {0021-7670},
 volume = {156},
 number = {1},
 pages = {213--251},
 year = {2025},
 language = {English},
 doi = {10.1007/s11854-025-0386-z},
 zbMATH = {8103492},
 Zbl = {1573.30064}
}

@article{DinhSibony2002,
  author  = {Dinh, Tien-Cuong and Sibony, Nessim},
  title   = {Sur les endomorphismes holomorphes permutables de {${\mathbb P}^k$}},
  journal = {Mathematische Annalen},
  volume  = {324},
  number  = {1},
  year    = {2002},
  pages   = {33--70},
  url     = {https://arxiv.org/abs/math/0007017}
}

@article{Kaufmann2018,
  author  = {Kaufmann, Lucas},
  title   = {Commuting pairs of endomorphisms of $\mathbb{P}^2$},
  journal = {Ergodic Theory and Dynamical Systems},
  volume  = {38},
  number  = {3},
  year    = {2018},
  pages   = {1025--1047},
  doi     = {10.1017/etds.2016.54}
}

@incollection{MilnorLattes,
  author        = {Milnor, John},
  title         = {On {L}att{\`e}s maps},
  booktitle     = {Dynamics on the Riemann Sphere},
  editor        = {Hjorth, Poul G. and Petersen, Carsten Lunde},
  publisher     = {European Mathematical Society},
  address       = {Z{\"u}rich},
  pages         = {9--43},
  year          = {2006},
  doi           = {10.4171/011-1/1},
  eprint        = {math/0402147},
  archivePrefix = {arXiv}
}

@article{LevyMorphisms,
  author  = {Levy, Alon},
  title   = {The space of morphisms on projective space},
  journal = {Acta Arithmetica},
  volume  = {146},
  number  = {1},
  year    = {2011},
  pages   = {13--31},
  doi     = {10.4064/aa146-1-2},
  url     = {https://doi.org/10.4064/aa146-1-2}
}

@article{SilvermanRatMaps,
  author  = {Silverman, Joseph H.},
  title   = {The space of rational maps on {$\mathbb{P}^1$}},
  journal = {Duke Mathematical Journal},
  volume  = {94},
  number  = {1},
  year    = {1998},
  pages   = {41--77},
  doi     = {10.1215/S0012-7094-98-09404-2},
  eprint  = {math/9609212},
  archivePrefix = {arXiv}
}

@article{SilvermanCommutingAffinePlane,
  author  = {Silverman, Joseph H.},
  title   = {Automorphism groups of commuting polynomial maps of the affine plane},
  journal = {Rend. Lincei Mat. Appl.},
  volume  = {36},
  pages   = {671--696},
  year    = {2025},
  doi     = {10.4171/RLM/1087}
}

@article{Fatou1923,
  author  = {Fatou, Pierre},
  title   = {Sur l'it{\'e}ration analytique et les substitutions permutables},
  journal = {Journal de Math{\'e}matiques Pures et Appliqu{\'e}es},
  series  = {9},
  volume  = {2},
  year    = {1923},
  pages   = {343--384},
  note    = {Continued in vol.~3 (1924), 1--50},
  url     = {http://eudml.org/doc/234679}
}

@article{Julia1922,
  author  = {Julia, Gaston},
  title   = {M{\'e}moire sur la permutabilit{\'e} des fractions rationnelles},
  journal = {Annales scientifiques de l'{\'E}cole Normale Sup{\'e}rieure},
  series  = {3},
  volume  = {39},
  year    = {1922},
  pages   = {131--215},
  doi     = {10.24033/asens.740}
}

@article{Eremenko1990,
  author  = {Eremenko, Alexandre},
  title   = {Some functional equations connected with the iteration of rational functions},
  journal = {Leningrad Mathematical Journal},
  volume  = {1},
  number  = {4},
  year    = {1990},
  pages   = {905--919},
  note    = {Translated from Algebra i Analiz \textbf{1} (1989), no.~4, 102--116}
}

@book{SGA1,
      author={Grothendieck, Alexander},
      title={Rev\^etements \'etales et groupe fondamental ({SGA 1})},
      series={Documents Math\'ematiques},
      volume={3},
      publisher={Soci\'et\'e Math\'ematique de France},
      address={Paris},
      year={2003},
      pages={xviii+327},
      note={S\'eminaire de G\'eom\'etrie Alg\'ebrique du Bois Marie 1960--61, dirig\'e par A. Grothendieck, augment\'e de deux expos\'es de Mich{\`e}le Raynaud}
}

@article{pcfsilverman,
 author = {Ingram, Patrick and Ramadas, Rohini and Silverman, Joseph H.},
 title = {Post-critically finite maps on {{\(\mathbb{P}^n\)}} for {{\(n\ge 2\)}} are sparse},
 fjournal = {Transactions of the American Mathematical Society},
 journal = {Trans. Am. Math. Soc.},
 issn = {0002-9947},
 volume = {376},
 number = {5},
 pages = {3087--3109},
 year = {2023},
 language = {English},
 doi = {10.1090/tran/8871},
 zbMATH = {7678796},
 Zbl = {1516.37146}
}

@article{GauthierTaflinVignyPCF,
  author    = {Gauthier, Thomas and Taflin, Johan and Vigny, Gabriel},
  title     = {Sparsity of postcritically finite maps of {$\mathbb{P}^k$}
               and beyond: {A} complex analytic approach},
  journal   = {Publications Math\'ematiques de l'IH\'ES},
  volume    = {143},
  pages     = {1--96},
  year      = {2026},
  doi       = {10.5802/pmihes.1},
  eprint    = {2305.02246},
  archivePrefix = {arXiv},
  primaryClass  = {math.DS}
}

@book{GuralnickLawther2024,
  author    = {Guralnick, Robert M. and Lawther, Ross},
  title     = {Generic stabilizers in actions of simple algebraic groups},
  series    = {Memoirs of the American Mathematical Society},
  volume    = {300},
  publisher = {American Mathematical Society},
  year      = {2024},
  doi       = {10.1090/memo/1502},
  note      = {No. 1502}
}

@unpublished{ZhangMeasureRigidity,
  author = {Zhang, Yugang},
  title  = {Measure rigidity for regular polynomial endomorphisms and applications},
  note   = {Preprint, arXiv:2608.23749},
  year   = {2026},
  eprint = {2608.23749},
  archivePrefix = {arXiv},
  primaryClass = {math.DS},
  url = {https://arxiv.org/abs/2608.23749}
}

@book{Liu2002,
  author    = {Liu, Qing},
  title     = {Algebraic Geometry and Arithmetic Curves},
  series    = {Oxford Graduate Texts in Mathematics},
  volume    = {6},
  note      = {Translated from the French by Reinie Ern{\'e}},
  publisher = {Oxford University Press},
  address   = {Oxford},
  year      = {2002},
  pages     = {xvi+576},
  isbn      = {978-0-19-850284-5},
  doi       = {10.1093/oso/9780198502845.001.0001}
}

@article{DeMarcoPCFDensity,
  author  = {DeMarco, Laura},
  title   = {Dynamical moduli spaces and elliptic curves},
  journal = {Annales de la Facult{\'e} des Sciences de Toulouse.
             Math{\'e}matiques},
  series  = {6},
  volume  = {27},
  number  = {2},
  year    = {2018},
  pages   = {389--420},
  doi     = {10.5802/afst.1573}
}

@article{GhiocaTucker2021,
  author  = {Ghioca, Dragos and Tucker, Thomas J.},
  title   = {A reformulation of the dynamical {Manin--Mumford} conjecture},
  journal = {Bulletin of the Australian Mathematical Society},
  volume  = {103},
  number  = {1},
  year    = {2021},
  pages   = {154--161},
  doi     = {10.1017/S0004972720000477}
}

@article{DeMarcoMavrakiPreperiodic,
  author        = {DeMarco, Laura and Mavraki, Niki Myrto},
  title         = {The geometry of preperiodic points in families of maps on {$\mathbb{P}^N$}},
  journal       = {Enseign. Math.},
  volume        = {72},
  pages         = {305--335},
  year          = {2026},
  doi           = {10.4171/LEM/1102},
  eprint        = {2407.10894},
  archivePrefix = {arXiv},
  primaryClass  = {math.DS},
  url           = {https://arxiv.org/abs/2407.10894}
}

@unpublished{DeMarcoCIRMLectures,
  author = {DeMarco, Laura},
  title  = {The (algebraic) geometry of preperiodic points in {$\mathbb{P}^N$}, in three lectures},
  note   = {Lecture notes for the {\'E}tats de la recherche {SMF} program on analytic, algebraic, and arithmetic dynamics, {CIRM}},
  year   = {2025},
  url    = {https://people.math.harvard.edu/~demarco/CIRM_Jan2025_LectureNotes.pdf}
}

@misc{GauthierTaflin2026,
  author = {Gauthier, Thomas and Taflin, Johan},
  title  = {Towards a uniform dynamical {Manin--Mumford} conjecture for polynomial endomorphisms of {$\mathbb{C}^k$}},
  note   = {Forthcoming},
  year   = {2026}
}

@book{Nekrashevych2005,
  author    = {Nekrashevych, Volodymyr},
  title     = {Self-Similar Groups},
  series    = {Mathematical Surveys and Monographs},
  volume    = {117},
  publisher = {American Mathematical Society},
  address   = {Providence, RI},
  year      = {2005},
  pages     = {xii+231},
  doi       = {10.1090/surv/117}
}

@incollection{Nekrashevych2011,
  author    = {Nekrashevych, Volodymyr},
  title     = {Iterated monodromy groups},
  booktitle = {Groups St Andrews 2009 in Bath. Volume 1},
  series    = {London Mathematical Society Lecture Note Series},
  volume    = {387},
  pages     = {41--93},
  publisher = {Cambridge University Press},
  address   = {Cambridge},
  year      = {2011},
  doi       = {10.1017/CBO9780511842467.004}
}

@article{Nekrashevych2014Models,
  author  = {Nekrashevych, Volodymyr},
  title   = {Combinatorial models of expanding dynamical systems},
  journal = {Ergodic Theory Dynam. Systems},
  volume  = {34},
  pages   = {938--985},
  year    = {2014},
  doi     = {10.1017/etds.2012.163}
}

@article{BartholdiNekrashevych2006,
  author  = {Bartholdi, Laurent and Nekrashevych, Volodymyr},
  title   = {Thurston equivalence of topological polynomials},
  journal = {Acta Math.},
  volume  = {197},
  number  = {1},
  pages   = {1--51},
  year    = {2006},
  doi     = {10.1007/s11511-006-0007-3}
}

@article{BartholdiNekrashevych2008,
  author  = {Bartholdi, Laurent and Nekrashevych, Volodymyr V.},
  title   = {Iterated monodromy groups of quadratic polynomials. {I}},
  journal = {Groups Geom. Dyn.},
  volume  = {2},
  number  = {3},
  pages   = {309--336},
  year    = {2008},
  doi     = {10.4171/GGD/42}
}

@article{Cogolludo2011,
  author  = {Cogolludo-Agust{\'\i}n, Jos{\'e} Ignacio},
  title   = {Braid monodromy of algebraic curves},
  journal = {Annales math{\'e}matiques Blaise Pascal},
  volume  = {18},
  number  = {1},
  pages   = {141--209},
  year    = {2011},
  doi     = {10.5802/ambp.295},
  url     = {https://www.numdam.org/articles/10.5802/ambp.295/}
}

@misc{LuoMeng2026,
  author        = {Luo, Yujie and Meng, Sheng},
  title         = {Log {Calabi--Yau} structure for endomorphisms on {$\mathbf{P}^n$}},
  year          = {2026},
  howpublished  = {Preprint, arXiv:\allowbreak 2608.02114},
  eprint        = {2608.02114},
  archivePrefix = {arXiv},
  primaryClass  = {math.AG},
  url           = {https://arxiv.org/abs/2608.02114}
}

@misc{OlechnowiczWeinreich2026,
  author        = {Olechnowicz, Matt and Weinreich, Max},
  title         = {Critical loci of self-maps of projective space},
  year          = {2026},
  howpublished  = {Preprint, arXiv:\allowbreak 2607.21923},
  eprint        = {2607.21923},
  archivePrefix = {arXiv},
  primaryClass  = {math.AG},
  url           = {https://arxiv.org/abs/2607.21923}
}

@misc{CantatXieBirationalConjugacies,
  author  = {Cantat, Serge and Xie, Junyi},
  title   = {Birational conjugacies between endomorphisms on the projective plane},
  year    = {2020},
  howpublished = {Preprint, arXiv:\allowbreak 2006.00051},
  url     = {https://arxiv.org/abs/2006.00051}
}

@article{JiXieMultiplierSpectrum,
  author  = {Ji, Zhuchao and Xie, Junyi},
  title   = {The multiplier spectrum morphism is generically injective},
  journal = {Journal of the European Mathematical Society},
  year    = {2025},
  note    = {Published online first},
  doi     = {10.4171/JEMS/1750}
}

@article{MedvedevScanlon2014,
  author  = {Medvedev, Alice and Scanlon, Thomas},
  title   = {Invariant varieties for polynomial dynamical systems},
  journal = {Annals of Mathematics},
  series  = {2},
  volume  = {179},
  number  = {1},
  year    = {2014},
  pages   = {81--177},
  doi     = {10.4007/annals.2014.179.1.2}
}

@article{PakovichInvariantCurves,
  author  = {Pakovich, Fedor},
  title   = {Invariant curves for endomorphisms of {$\mathbb P^1\times\mathbb P^1$}},
  journal = {Mathematische Annalen},
  volume  = {385},
  year    = {2023},
  pages   = {259--307},
  doi     = {10.1007/s00208-021-02304-5}
}

@article{PakovichSemiconjugate2016,
  author  = {Pakovich, Fedor},
  title   = {On semiconjugate rational functions},
  journal = {Geometric and Functional Analysis},
  volume  = {26},
  number  = {4},
  year    = {2016},
  pages   = {1217--1243},
  doi     = {10.1007/s00039-016-0383-6}
}
\end{document}